\documentclass[lettersize,journal]{IEEEtran}
\usepackage{amsmath,amsfonts,amssymb}
\usepackage{algorithmic}
\usepackage{algorithm}

\usepackage{array}
\usepackage{textcomp}
\usepackage{stfloats}
\usepackage{url}
\usepackage{verbatim}
\usepackage{graphicx}
\usepackage{cite}
\usepackage{hyperref}
\ifCLASSOPTIONcompsoc
  \usepackage[caption=false,font=normalsize,labelfont=sf,textfont=sf]{subfig}
\else
  \usepackage[caption=false,font=footnotesize]{subfig}
\fi

\newtheorem{theorem}{Theorem}
\newtheorem{lemma}{Lemma}
\newtheorem{assumption}{Assumption}

\newtheorem{remark}{Remark}

\begin{document}

\title{Distributed Zeroth-Order Optimization with Rademacher Perturbations and Momentum Gradient Tracking}

\author{Yanxu~Su,~\IEEEmembership{Member,~IEEE,} Xiaorui Tong, 
	and~Changyin~Sun,~\IEEEmembership{Senior Member,~IEEE}
	\thanks{The authors are with the School of Artificial Intelligence, Anhui University, Hefei 230601, China, and also with the Engineering Research Center of Autonomous Unmanned System Technology, Ministry of Education, Hefei 230601, China (e-mail: yanxu.su@ahu.edu.cn, W125211087@stu.ahu.edu.cn, cysun@ahu.edu.cn).}}

\markboth{Journal of \LaTeX\ Class Files,~Vol.~14, No.~8, August~2021}%
{Shell \MakeLowercase{\textit{et al.}}: A Sample Article Using IEEEtran.cls for IEEE Journals}


\maketitle

\begin{abstract}

Zeroth-order (ZO) optimization is indispensable for complex non-convex tasks where explicit gradients are computationally prohibitive or strictly inaccessible. For deploying ZO methods over distributed heterogeneous networks, the gradient tracking technique is often employed to eliminate structural data biases. However, the inherent variance of derivative-free estimators is also amplified. To overcome this problem, we propose Zeroth-Order Momentum Gradient Tracking (ZO-MGT), which integrates momentum-based variance reduction with dynamic gradient tracking. Specifically, ZO-MGT that requires exactly two function queries per iteration can avoid costly batch sampling and prevent variance explosion, while eliminating structural biases. Moreover, by utilizing Rademacher perturbations, it preserves optimal query efficiency and enables bitwise hardware acceleration. We theoretically analyze the convergence of ZO-MGT and establish an $\mathcal{O}(1/T)$ convergence rate. Furthermore, we prove that a large momentum factor can aggressively suppress the heterogeneity-induced bias floor at a remarkable quadratic rate of $\mathcal{O}((1-\beta)^2)$. Numerical experiments under extreme data heterogeneity verify that ZO-MGT can effectively overcome traditional tracking failures with accelerated convergence guarantees, while achieving significantly tighter consensus.

\end{abstract}


\section{Introduction}

\IEEEPARstart{D}{istributed} optimization has emerged as a fundamental paradigm for addressing large-scale machine learning, sensor networks, and multi-agent control problems. In a standard distributed setting, a network of agents collaborates to minimize a global objective function using only local computations and peer-to-peer communications \cite{nedicDistributedSubgradientMethods2009}. Mathematically, this cooperative task is typically formulated as the following finite-sum minimization problem:
\begin{equation}\label{DOP}
	\min_{x \in \mathbb{R}^d} F(x) = \frac{1}{N} \sum_{i=1}^N f_i(x), 
\end{equation}
where $N$ is the total number of agents and $f_i(x)$ denotes the local cost function accessible only to agent $i$. Over the past decade, numerous gradient-based distributed algorithms have been proposed for the problem in \eqref{DOP}. Classic first-order methods, e.g., Distributed Gradient Descent (DGD) \cite{nedicDistributedSubgradientMethods2009}, establish early foundations but suffer from non-vanishing steady-state bias under constant step sizes in highly heterogeneous networks. To eliminate the steady-state bias and guarantee exact consensus, various frameworks have been proposed. Early approaches, including the Alternating Direction Method of Multipliers (ADMM) \cite{shi2014linear} and distributed augmented Lagrangian methods \cite{jakovetic2015linear}, appeal to dual variable updates and then develop a double-loop update for primal and dual variables. To simplify the updating to single-loop, subsequent research has focused on exact first-order methods, prominently including EXTRA \cite{shiEXTRAExactFirstOrder2015}, Exact Diffusion \cite{yuan2018exact}, and Gradient Tracking (GT) \cite{quHarnessingSmoothnessAccelerate2018, nedicAchievingGeometricConvergence2017}. While all these methods can guarantee exact linear convergence for smooth and strongly convex objectives, GT is particularly notable for its decoupled tracking structure. By employing dynamic average consensus to track the global gradient, GT provides a flexible architecture that is readily adaptable to non-convex extensions and complex time-varying topologies \cite{puDistributedStochasticGradient2021}. However, despite their theoretical appeal, the applicability of these exact first-order methods is fundamentally constrained by the restrictive assumption, i.e., analytical gradients are readily available and computationally inexpensive to evaluate.

In many emerging applications, the analytical gradients of the objective functions are either computationally prohibitive to evaluate or entirely unavailable \cite{liuPrimerZerothorderOptimization2020, nesterovRandomGradientfreeMinimization2017}.
A prominent example is the parameter-efficient fine-tuning of Large Language Models (LLMs), where the backward pass of first-order optimizers imposes excessive GPU memory overhead \cite{malladiFinetuningLanguageModels2023, zhaoSecondorderFinetuningPain2025, yeWhyDoesAdaptive2026}. Furthermore, explicit gradients are fundamentally inaccessible in black-box adversarial attacks on deep neural networks \cite{chenZOOZerothOrder2017} and simulation-based reinforcement learning \cite{fazelGlobalConvergencePolicy2018}. To circumvent this reliance on explicit derivatives, zeroth-order (ZO) optimization has emerged as a promising alternative. ZO methods approximate the true gradient by querying only the forward function values along random perturbation directions. Unlike traditional ZO methods based on continuous Gaussian smoothing, recent studies such as FZOO \cite{dangFZOOFastZerothorder2025} demonstrate the computational efficiency of discrete Rademacher sampling ($u \in \{-1, 1\}^d$). By replacing floating-point Gaussian vectors with discrete Rademacher perturbations, ZO estimators maintain comparable approximation accuracy while significantly reducing computational overhead. Moreover, this discrete sampling mechanism facilitates efficient bitwise operations, paving the way for hardware-level acceleration in future deployments.

Leveraging its derivative-free nature, ZO optimization has been widely extended to distributed and decentralized networks. For example, the distributed ZO algorithms proposed by Tang \textit{et al.} \cite{tangDistributedZeroorderAlgorithms2021} established the initial convergence guarantees for non-convex multi-agent optimization. Building upon this, Yi \textit{et al.} \cite{yiZerothorderAlgorithmsStochastic2022} demonstrated a linear speedup for stochastic distributed ZO methods under relaxed, state-dependent variance assumptions. Recently, the aforementioned theoretical frameworks have been tailored to specific network constraints, enabling communication-efficient federated learning \cite{fangCommunicationefficientStochasticZerothorder2022} and deployment in resource-constrained edge Internet of Things (IoT) systems \cite{dangAdaptiveCommunicationefficientZerothorder2024}.

Despite the advances, practical deployment of distributed ZO algorithms is bottlenecked by the inherent coupling of data heterogeneity and estimator variance. In highly heterogeneous environments, standard ZO methods utilizing basic random estimators fail to effectively suppress the steady-state bias floor induced by data heterogeneity. Although recent studies have explored using gradient-free feedback \cite{jinZerothorderFeedbackbasedOptimization2023} or communication compression \cite{wangHeterogeneousDistributedZerothorder2026}, achieving highly accurate tracking without restrictive assumptions, such as massive batch sizes or data homogeneity, still remains a critical open problem.

Integrating GT into ZO frameworks can theoretically eliminate heterogeneity-induced bias, which motivates recent extensions for unbalanced graphs \cite{wang2025distributed} and compressed networks \cite{xu2024compressed}. However, applying GT directly to ZO optimization faces critical practical limitations. GT fundamentally relies on accumulating successive gradient differences to track the global direction. Due to the high sampling variance of derivative-free estimators, these consecutive differences are easily overwhelmed by noise. For instance, using single-point estimates to reduce query overhead \cite{mhanna2023single} introduces severe inherent variance. The GT mechanism cumulatively amplifies this kind of noise, critically destabilizing network consensus \cite{linDecentralizedGradientfreeMethods2024}.

To achieve accurate tracking without amplifying variance, various variance reduction techniques have been developed. However, recent studies highlight a fundamental trade-off between sampling costs and convergence speeds in distributed ZO optimization \cite{mu2025variance}. Existing methods typically face a dilemma between computational and communication bottlenecks. To suppress sampling variance, many approaches rely on deterministic full-coordinate estimators \cite{tangDistributedZeroorderAlgorithms2021} or standard stochastic variance reduction methods, such as ZO-SVRG and ZO-SPIDER \cite{liu2018zeroth, ji2019improved}. Because both strategies require $\mathcal{O}(d)$ queries per iteration, they impose a severe computational burden that diminishes the core advantage of derivative-free optimization. Conversely, while incremental variance reduction frameworks \cite{zhangZIVRIncrementalVarianceReduction2026} alleviate the reliance on large batch sizes, their fully decentralized implementation incurs substantial communication overhead.

To overcome these prohibitive costs, momentum serves as a computationally efficient alternative. Cutkosky and Orabona \cite{cutkoskyMomentumbasedVarianceReduction2019a} established that momentum acts as an implicit variance reducer in centralized non-convex optimization, achieving optimal convergence without massive batch sizes. Leveraging this property, momentum has been increasingly integrated into ZO methods to smooth stochastic gradient estimates \cite{chenZOAdaMMZerothorderAdaptive2019}, derive query complexity lower bounds \cite{huangAcceleratedZerothorderFirstorder2022}, and stabilize adaptive fine-tuning \cite{yeWhyDoesAdaptive2026}. Nevertheless, extending momentum-based variance reduction to decentralized heterogeneous networks presents significant challenges. Although recent studies use momentum in distributed ZO \cite{dangMomentumbasedZerothorderGradient2026}, they depend on centralized aggregation rather than decentralized peer-to-peer architectures. In environments with significant data heterogeneity, it remains a critical open problem to effectively integrate momentum-based variance reduction and dynamic gradient tracking to reduce sampling noise and correct heterogeneity-induced bias, respectively.

In this paper, we propose the Zeroth-Order Momentum Gradient Tracking (ZO-MGT) algorithm for distributed non-convex optimization, aiming to address the aforementioned problems. To be specific, by integrating momentum-based variance reduction with dynamic gradient tracking, ZO-MGT ensures stable and efficient convergence even under substantial data heterogeneity. The main contributions of this work are summarized as follows.

\begin{enumerate}
	\item We design a derivative-free distributed optimization framework that effectively eliminates the steady-state bias caused by the extreme data heterogeneity. The proposed ZO algorithm requires two function queries per iteration and supports Rademacher random perturbations, thereby laying the algorithmic groundwork for future hardware-level acceleration via efficient bitwise operations.
	
	\item We establish a rigorous non-convex convergence analysis via the Lyapunov stability theory, thereby providing an $\mathcal{O}(1/T)$ convergence rate. Crucially, we theoretically reveal that momentum suppresses the heterogeneity-induced bias floor at a quadratic rate of $\mathcal{O}((1-\beta)^2)$, which completely avoids the variance explosion common in standard ZO tracking.
	
	\item We validate the profound empirical superiority of our approach through extensive experiments on distributed non-convex problems under extreme heterogeneous settings. Compared to existing representative ZO baselines, ZO-MGT achieves orders-of-magnitude tighter consensus and significantly accelerates the convergence.
\end{enumerate}

The remainder of this paper is organized as follows. Section II defines the distributed optimization problem and provides a detailed description of the proposed ZO-MGT algorithm. In Section III, we present the preliminary results for the theoretical analysis, thereby providing a comprehensive Lyapunov-based convergence analysis and the steady-state bias floor characterization. Section IV validates the effectiveness and efficiency of the proposed framework through some numerical experiments on non-convex classification tasks. Finally, Section V concludes the paper and discusses potential directions for future research.

\section{Problem and Algorithm}

\subsection{Notation}
The following notation is adopted throughout this paper. Let $\mathcal{V} \triangleq \{1, \dots, N\}$ denote the set of agents. For each agent $i \in \mathcal{V}$ at iteration $k$, the local variables are $x_{k,i}, m_{k,i}, y_{k,i} \in \mathbb{R}^d$, and their aggregate row-wise matrices are $\mathbf{X}_k, \mathbf{M}_k, \mathbf{Y}_k \in \mathbb{R}^{N \times d}$, respectively. We define $\bar{z}_k \triangleq \frac{1}{N} \sum_{i=1}^N z_{k,i}$ as the global average of any local variable $z \in \{x, m, y, \hat{g}\}$. Specifically, the averaging operator is defined as $J \triangleq \frac{1}{N}\mathbf{11}^\top$, such that for any matrix $\mathbf{A} \in \mathbb{R}^{N \times d}$, the product $J\mathbf{A} = \mathbf{1}\bar{a}^\top$ projects each row onto the average of its components. Regarding norms, we use $\|\cdot\|$ to denote the Euclidean norm for vectors, while $\|\cdot\|_F$ and $\|\cdot\|_2$ represent the Frobenius and spectral norms for matrices, respectively. The network-wide consensus error is quantified by $\Xi_k \triangleq \frac{1}{N} \|(I-J)\mathbf{X}_k\|_F^2$.

The communication topology is modeled by an undirected graph $\mathcal{G}$ and its associated mixing matrix $W \in \mathbb{R}^{N \times N}$. A key parameter in our convergence analysis is the spectral gap $\rho \triangleq \|W - J\|_2$, which satisfies $\rho \in [0, 1)$ for any connected graph. This value measures the network connectivity and determines the speed of information diffusion. Essentially, a smaller $\rho$ represents a more well-connected topology, which promotes faster consensus and more efficient information exchange among agents.

\subsection{Problem Formulation}

Consider a network of $N$ agents collaborating to solve a global optimization problem. The objective is to minimize the average of local cost functions, formulated as:
\begin{equation}
\min_{x \in \mathbb{R}^d} F(x) = \frac{1}{N} \sum_{i=1}^N f_i(x),
\label{eq:global_obj}
\end{equation}
where $f_i: \mathbb{R}^d \to \mathbb{R}$ is a smooth, non-convex cost function accessible only to agent $i$. In the derivative-free setting, the analytical gradient $\nabla f_i(x)$ is unavailable, and agents can only probe the local objective by querying function values.

Unlike centralized optimization, the agents must solve the problem in \eqref{eq:global_obj} cooperatively through local information exchange. The communication topology is modeled by an undirected graph $\mathcal{G}=\{\mathcal{V},\mathcal{E}\}$, where $\mathcal{V}=\{1,2,\dots,N\}$ is the set of agents and $\mathcal{E}\subseteq\mathcal{V}\times\mathcal{V}$ is the set of communication links. A node pair $\{i,j\}\in\mathcal{E}$ if and only if agents $i$ and $j$ can directly communicate. 

To facilitate this distributed information exchange, a mixing matrix $W = [w_{ij}] \in \mathbb{R}^{N \times N}$ is associated with the graph $\mathcal{G}$. The element $w_{ij}$ denotes the consensus weight that agent $i$ assigns to the information received from agent $j$. 

To formally analyze the convergence of the proposed algorithm, we make the following standard assumptions regarding the objective functions and the communication network topology.

\begin{assumption}
\label{ass:smooth}
The local cost function $f_i$ is $L$-smooth. There exists a constant $L > 0$ such that for all $x, y \in \mathbb{R}^d$ and all $i \in \mathcal{V}$:
\begin{equation}
\| \nabla f_i(x) - \nabla f_i(y) \| \le L \| x - y \|.
\end{equation}
\end{assumption}

\begin{assumption}
\label{ass:lower_bound}
The global objective function $F(x)$ is lower bounded, i.e., there exists a constant $F^\star > -\infty$ such that $F(x) \ge F^\star$ for all $x \in \mathbb{R}^d$.
\end{assumption}

\begin{assumption}
\label{ass:graph}
The communication graph $\mathcal{G}$ is undirected and connected. The mixing matrix $W = [w_{ij}] \in \mathbb{R}^{N \times N}$ satisfies the following properties:
\begin{enumerate}
    \item $w_{ij} > 0$ if $\{i,j\} \in \mathcal{E}$ or $i=j$; otherwise, $w_{ij} = 0$.
    \item $W \mathbf{1} = \mathbf{1}$ and $\mathbf{1}^\top W = \mathbf{1}^\top$.
    \item There exists a constant $\rho \in [0, 1)$ such that
    \begin{equation}
        \left\| W - \frac{1}{N}\mathbf{1}\mathbf{1}^\top \right\|_2 \le \rho.
    \end{equation}
\end{enumerate}
\end{assumption}

\begin{assumption}
\label{ass:heterogeneity}
To measure the data heterogeneity across the network, we assume the variance of local gradients is bounded by a constant $\zeta > 0$:
\begin{equation}
\frac{1}{N} \sum_{i=1}^N \|\nabla f_i(x) - \nabla F(x)\|^2 \le \zeta^2.
\end{equation}
\end{assumption}

\begin{remark}
The aforementioned assumptions are standardized within the literature of decentralized and derivative-free optimization \cite{mu2025variance, puDistributedStochasticGradient2021}. 
While traditional zeroth-order (ZO) frameworks frequently impose more restrictive constraints—such as requiring $O(d)$ coordinate-wise queries per iteration to achieve variance reduction or assuming data homogeneity to ensure stability \cite{mu2025variance},our analysis relies on significantly milder premises. 
In practical deployments, these conditions are widely satisfied, as $L$-smoothness is a characteristic of common objectives, and network connectivity is a fundamental requirement for multi-agent coordination.
\end{remark}

\subsection{Algorithm Description}
\label{sec:algorithm}

\begin{algorithm}[htbp]
\caption{ZO-MGT (Zeroth-Order Momentum Gradient Tracking)}
\label{alg:zo_mgt}
\begin{algorithmic}[1]
\REQUIRE Parameters $x \in \mathbb{R}^d$, local loss $f_i: \mathbb{R}^d \to \mathbb{R}$, iteration budget $T$, perturbation scale $\mu$, step size $\eta$, momentum factor $\beta$, mixing matrix $W$.

\STATE \textbf{Initialize:} $x_{0,i}$ arbitrarily for all $i \in \mathcal{V}$.

\STATE \textbf{Pre-computation}
    \STATE \hspace{\algorithmicindent} $\hat{g}_{0,i} \leftarrow \textsc{GradEst}(x_{0,i}, \mu, f_i)$
    \STATE \hspace{\algorithmicindent} $m_{0,i} \leftarrow \hat{g}_{0,i}$
    \STATE \hspace{\algorithmicindent} $y_{0,i} \leftarrow m_{0,i}$

\FOR{$k =  1, \dots, T$}
    \FOR{each agent $i \in \mathcal{V}$ \textbf{in parallel}}
        \STATE $x_{k,i} \leftarrow \sum_{j=1}^N w_{ij} x_{k-1,j} - \eta y_{k-1,i}$ 
        
        \STATE $\hat{g}_{k,i} \leftarrow \textsc{GradEst}(x_{k,i}, \mu, f_i)$ 
        
        \STATE $m_{k,i} \leftarrow \beta m_{k-1,i} + (1-\beta)\hat{g}_{k,i}$ 

        \STATE $y_{k,i} \leftarrow \sum_{j=1}^N w_{ij} y_{k-1,j} + (m_{k,i} - m_{k-1,i})$ 
    \ENDFOR
\ENDFOR

\STATE \textbf{function} $\textsc{GradEst}(x, \mu, f)$
    \STATE \hspace{\algorithmicindent} Sample $u \sim \text{Unif}(\{-1, 1\}^d)$ 
    \STATE \hspace{\algorithmicindent} $\hat{g} \leftarrow \frac{1}{\mu} [f(x + \mu u) - f(x)] u$
    \STATE \hspace{\algorithmicindent} return $\hat{g}$
\STATE \textbf{end function}
\end{algorithmic}
\end{algorithm}

To address the distributed non-convex optimization problem in a derivative-free setting, we propose the Zeroth-Order Momentum Gradient Tracking (ZO-MGT) algorithm. Algorithm \ref{alg:zo_mgt} outlines its pseudo-code. Our proposed ZO-MGT seamlessly integrates two-point gradient estimation, momentum variance reduction, and gradient tracking.

\textit{1) Zeroth-Order Estimation and Momentum Variance Reduction:} Unlike first-order methods, we consider the case that analytical gradients are unavailable to the agents in this paper. Each agent $i$ queries the objective function using a random Rademacher perturbation vector $u$ via the $\textsc{GradEst}(x_{k,i}, \mu, f_i)$ subroutine to construct a highly efficient, two-point forward-difference gradient estimator $\hat{g}_{k,i}$. Because this zeroth-order estimator inherently introduces significant sampling variance, ZO-MGT incorporates a variance reduction technique. In the momentum update step, agent $i$ computes a momentum sequence $m_{k,i}$ by using a moving average parameter $\beta \in (0, 1)$. This exponential moving average effectively filters out the sampling noise and stabilizes the local descent trajectory over time.

\textit{2) Dynamic Consensus and Gradient Tracking:} To tackle the steady-state bias induced by data heterogeneity, ZO-MGT relies on dynamic average consensus and gradient tracking technique. In the state update step, agent $i$ aggregates the variables from its neighbors to achieve network consensus, while simultaneously taking a descent step along the historical tracked direction $y_{k-1,i}$. Furthermore, the tracking update step adjusts the variable $y_{k,i}$ to track the global average momentum by aggregating the tracking values of neighboring nodes and combining them with the local momentum difference. This critical mechanism dynamically corrects the local descent direction and mitigates the steady-state bias caused by severe data heterogeneity.

The parameters $\eta$, $\beta$, and $\mu$ directly affect the convergence speed and stability of the algorithm. To facilitate the subsequent theoretical analysis, we define the matrix forms $\mathbf{X}_k, \mathbf{G}_k, \mathbf{M}_k, \mathbf{Y}_k \in \mathbb{R}^{N \times d}$ by collecting the local variables of all $N$ agents. The compact form of Algorithm \ref{alg:zo_mgt} is given as follows:

\begin{subequations}
\begin{align}
    \mathbf{X}_k &= W \mathbf{X}_{k-1} - \eta \mathbf{Y}_{k-1} \\
    \mathbf{G}_k &= \textsc{GradEst}(\mathbf{X}_k, \mu, f) \\
    \mathbf{M}_k &= \beta \mathbf{M}_{k-1} + (1-\beta) \mathbf{G}_k \\
    \mathbf{Y}_k &= W \mathbf{Y}_{k-1} + \mathbf{M}_k - \mathbf{M}_{k-1}
\end{align}
\end{subequations}

\section{Theoretical Analysis}
\label{sec:analysis}

In this section, we establish the theoretical convergence guarantees of the proposed ZO-MGT algorithm. To disentangle the coupled errors in our distributed zeroth-order setting, our analysis proceeds in several logical steps. First, we characterize the exact evolution of the global average state (Lemma \ref{lem:global_dynamics}) and bound the statistical properties---specifically the bias and variance---of the Rademacher-based two-point gradient estimator (Lemma \ref{lem:zo_properties}). Next, we analyze the core variance reduction mechanism by deriving a recursive bound for the momentum tracking error (Lemma \ref{lem:momentum_error}). To address the decentralized nature of the algorithm, we then bound the network consensus bias induced by data heterogeneity (Lemma \ref{lem:consensus_recursion}). Finally, by evaluating the global objective descent (Lemma \ref{lem:descent_lemma}), we consolidate these bounds into a unified Lyapunov framework to establish the exact $\mathcal{O}(1/\sqrt{NK})$ convergence rate of ZO-MGT (Theorem \ref{thm:convergence}).

\subsection{Preliminaries and Auxiliary Variables}

\begin{lemma}
\label{lem:global_dynamics}
Suppose the communication network topology satisfy Assumption \ref{ass:graph}. Under the exact initialization of Algorithm \ref{alg:zo_mgt} where $y_{0,i} = m_{0,i}$ for all agents $i \in \{1, \dots, N\}$, the global average sequences evolve according to the following centralized dynamics for all $k \ge 1$:
\begin{align}
    \bar{x}_k &= \bar{x}_{k-1} - \eta \bar{y}_{k-1}, \label{eq:global_x_update} \\
    \bar{y}_k &= \bar{m}_k = \frac{1}{N} \sum_{i=1}^N m_{k,i}. \label{eq:global_y_conservation}
\end{align}
\end{lemma}

\begin{IEEEproof}
To establish the exact evolution of the global average state, we begin by analyzing the local parameter update rule for an arbitrary agent $i$:
\begin{equation}
    x_{k,i} = \sum_{j=1}^N w_{ij} x_{k-1,j} - \eta y_{k-1,i}.
\end{equation}
Summing this local update rule over all $N$ agents yields:
\begin{equation} \label{eq:sum_x_local}
    \sum_{i=1}^N x_{k,i} = \sum_{i=1}^N \left( \sum_{j=1}^N w_{ij} x_{k-1,j} \right) - \eta \sum_{i=1}^N y_{k-1,i}.
\end{equation}
By interchanging the order of the double summation, we group the consensus weights associated with the historical state $x_{k-1,j}$:
\begin{equation} \label{eq:double_sum_swap}
    \sum_{i=1}^N \sum_{j=1}^N w_{ij} x_{k-1,j} = \sum_{j=1}^N x_{k-1,j} \left( \sum_{i=1}^N w_{ij} \right).
\end{equation}
According to Assumption \ref{ass:graph}, the mixing matrix $W$ is doubly stochastic, which explicitly implies it is column-stochastic, i.e., $\sum_{i=1}^N w_{ij} = 1$ for all $j$. Applying this property to \eqref{eq:double_sum_swap}, the sum simplifies exactly to $\sum_{j=1}^N x_{k-1,j}$. Substituting this back into \eqref{eq:sum_x_local} and dividing both sides by the network size $N$, we directly obtain the global variable update:
\begin{equation}
    \bar{x}_k = \bar{x}_{k-1} - \eta \bar{y}_{k-1}.
\end{equation}

Next, we evaluate the gradient tracking sequence $\bar{y}_k$. By definition, the local tracking update is:
\begin{equation}
    y_{k,i} = \sum_{j=1}^N w_{ij} y_{k-1,j} + m_{k,i} - m_{k-1,i}.
\end{equation}
Following the exact same double-summation expansion and utilizing the column-stochastic property of $W$, summing this update over all $N$ agents and dividing by $N$ yields the macroscopic tracking recursion:
\begin{equation} \label{eq:y_recursion}
    \bar{y}_k = \bar{y}_{k-1} + \bar{m}_k - \bar{m}_{k-1}.
\end{equation}
To prove $\bar{y}_k = \bar{m}_k$ for all $k \ge 1$, we proceed by mathematical induction. For the case $k=0$, the explicit algorithmic initialization $y_{0,i} = m_{0,i}$ directly guarantees $\bar{y}_0 = \bar{m}_0$. Rearranging the recursion \eqref{eq:y_recursion} as $\bar{y}_k - \bar{m}_k = \bar{y}_{k-1} - \bar{m}_{k-1}$, we observe that the difference between the tracking sequence and the momentum sequence remains strictly constant across all iterations. Since the initial difference is $\bar{y}_0 - \bar{m}_0 = 0$, it universally follows that $\bar{y}_k = \bar{m}_k$ for all $k \ge 1$. This completes the proof.
\end{IEEEproof}

Unlike first-order methods, ZO-MGT relies entirely on derivative-free evaluations. Consequently, the gradient information injected into the global dynamics is inherently perturbed. Before analyzing the tracking mechanism, we must rigorously formalize the statistical properties, specifically the smoothing bias and variance bounds of the Rademacher-based two-point gradient estimator.

\begin{lemma}
\label{lem:zo_properties}
Suppose the local objective function $f_i$ is $L$-smooth. The two-point gradient estimator $\hat{g}_{k,i}$ which is defined in Algorithm \ref{alg:zo_mgt} with Rademacher sampling satisfies the following bias and variance bounds:
\begin{align}
    \left\| \mathbb{E}_{u}[\hat{g}_{k,i}] - \nabla f_i(x_{k,i}) \right\| \le \frac{\mu L d^{1.5}}{2} \triangleq \delta_\mu, \label{eq:bias_bound} \\
    \mathbb{E}_{u}\left[ \|\hat{g}_{k,i}\|^2 \right] \le 2d \|\nabla f_i(x_{k,i})\|^2 + \frac{\mu^2 L^2 d^3}{2} \triangleq \sigma_{k,i}^2. \label{eq:second_moment}
\end{align}
Thus, the true variance of the estimator is strictly bounded by the same second-moment quantity:
\begin{equation}
    \mathbb{E}_{u}\left[ \|\hat{g}_{k,i} - \mathbb{E}_{u}[\hat{g}_{k,i}]\|^2 \right] \le \sigma_{k,i}^2. \label{eq:variance_bound}
\end{equation}
\end{lemma}

\begin{IEEEproof}
We rigorously analyze the statistical properties of the local zeroth-order gradient estimator for an arbitrary agent $i$ at iteration $k$. By the definition of $\textsc{GradEst}$, the forward-difference estimator is given by
\begin{equation}
	\hat{g}_{k,i} = \frac{f_i(x_{k,i}+\mu u_{k,i}) - f_i(x_{k,i})}{\mu} u_{k,i}. 
\end{equation}
The random perturbation vector $u_{k,i}$ is drawn from a Rademacher distribution, which inherently satisfies the following statistical identities: $\mathbb{E}[u_{k,i}] = \mathbf{0}$, $\mathbb{E}[u_{k,i}u_{k,i}^\top] = I_d$, and $\|u_{k,i}\|^2 = d$.

To establish the bias bound \eqref{eq:bias_bound}, we rewrite the function difference using the fundamental theorem of calculus as an integral along the perturbation direction:
\begin{equation}
    f_i(x_{k,i}+\mu u_{k,i}) - f_i(x_{k,i}) = \mu \int_{0}^{1} \nabla f_i(x_{k,i} + t\mu u_{k,i})^\top u_{k,i} \, dt.
\end{equation}
Substituting this integral back into the definition of the estimator yields:
\begin{equation} \label{eq:g_integral}
    \hat{g}_{k,i} = \left( \int_{0}^{1} \nabla f_i(x_{k,i} + t\mu u_{k,i})^\top u_{k,i} \, dt \right) u_{k,i}.
\end{equation}
Taking the expectation with respect to $u_{k,i}$, we can express the true gradient using the identity $\nabla f_i(x_{k,i}) = \mathbb{E}[u_{k,i} u_{k,i}^\top] \nabla f_i(x_{k,i}) = \mathbb{E}[u_{k,i} u_{k,i}^\top \nabla f_i(x_{k,i})]$. Subtracting this from the expectation of \eqref{eq:g_integral} allows us to group the terms inside a unified integral:
\begin{align} \label{eq:bias_integral_form}
    &\mathbb{E}[\hat{g}_{k,i}] - \nabla f_i(x_{k,i}) \nonumber \\
    &= \mathbb{E}\left[ \int_{0}^{1} u_{k,i} u_{k,i}^\top \Big( \nabla f_i(x_{k,i} + t\mu u_{k,i}) - \nabla f_i(x_{k,i}) \Big) \, dt \right].
\end{align}
Taking the norm on both sides, we apply Jensen's inequality to move the norm inside both the expectation and the integral. Furthermore, we utilize the property that $u_{k,i} u_{k,i}^\top$ is a rank-1 matrix with a non-zero eigenvalue of $\|u_{k,i}\|^2 = d$, which directly implies its spectral norm is $\|u_{k,i} u_{k,i}^\top\|_2 = d$. Applying the $L$-smoothness assumption ($\|\nabla f(x) - \nabla f(y)\| \le L\|x-y\|$), the bias is explicitly bounded by:
\begin{align}
    \|\mathbb{E}[\hat{g}_{k,i}] - \nabla f_i(x_{k,i})\| 
    &\le \mathbb{E}\left[ \int_{0}^{1} \|u_{k,i} u_{k,i}^\top\|_2 \cdot L \|t\mu u_{k,i}\| \, dt \right] \nonumber \\
    &= \mathbb{E}\left[ d \cdot L \mu \sqrt{d} \int_{0}^{1} t \, dt \right] \nonumber \\
    &= d \cdot L \mu \sqrt{d} \left( \frac{1}{2} \right) = \frac{\mu L d^{1.5}}{2}.
\end{align}

To derive the second moment bound \eqref{eq:second_moment}, we utilize the squared norm of the estimator: 
\begin{equation}
	\|\hat{g}_{k,i}\|^2 = \frac{d}{\mu^2} \left( f_i(x_{k,i}+\mu u_{k,i}) - f_i(x_{k,i}) \right)^2.	
\end{equation}
By adding and subtracting the directional derivative $\mu \nabla f_i(x_{k,i})^\top u_{k,i}$ inside the squared term, and applying the basic algebraic inequality $(a+b)^2 \le 2a^2 + 2b^2$, we separate the true gradient projection from the non-linear approximation error:
\begin{align} \label{eq:squared_estimator_split}
    \|\hat{g}_{k,i}\|^2 &\le 2d \Big( \nabla f_i(x_{k,i})^\top u_{k,i} \Big)^2 + \frac{2d}{\mu^2} \Big( f_i(x_{k,i}+\mu u_{k,i}) \nonumber \\
    &\quad - f_i(x_{k,i}) - \mu \nabla f_i(x_{k,i})^\top u_{k,i} \Big)^2.
\end{align}
According to the standard quadratic bound derived from $L$-smoothness, the absolute approximation error is strictly bounded by $\frac{L}{2} \|\mu u_{k,i}\|^2 = \frac{L \mu^2 d}{2}$. Substituting it into \eqref{eq:squared_estimator_split} yields:
\begin{equation} \label{eq:squared_bound_pre_exp}
    \|\hat{g}_{k,i}\|^2 \le 2d \Big( \nabla f_i(x_{k,i})^\top u_{k,i} \Big)^2 + 2d \left( \frac{L \mu^2 d}{2 \mu} \right)^2.
\end{equation}
Taking the expectation with respect to $u_{k,i}$, the first term is evaluated by expressing the squared inner product as a quadratic form: 
\begin{align} \label{eq:expectation_quadratic}
    \mathbb{E}\Big[ (\nabla f_i(x_{k,i})^\top u_{k,i})^2 \Big] &= \mathbb{E}\Big[ \nabla f_i(x_{k,i})^\top u_{k,i} u_{k,i}^\top \nabla f_i(x_{k,i}) \Big] \nonumber \\
    &= \nabla f_i(x_{k,i})^\top \mathbb{E}[u_{k,i} u_{k,i}^\top] \nabla f_i(x_{k,i}) \nonumber \\
    &= \nabla f_i(x_{k,i})^\top I_d \nabla f_i(x_{k,i}) \nonumber \\
    &= \|\nabla f_i(x_{k,i})\|^2.
\end{align}
Substituting \eqref{eq:expectation_quadratic} back into the expectation of \eqref{eq:squared_bound_pre_exp} completely resolves the second moment bound:
\begin{equation}
    \mathbb{E}[\|\hat{g}_{k,i}\|^2] \le 2d \|\nabla f_i(x_{k,i})\|^2 + \frac{\mu^2 L^2 d^3}{2} = \sigma_{k,i}^2.
\end{equation}

Finally, to establish the variance bound \eqref{eq:variance_bound}, we apply the fundamental statistical variance decomposition $\mathbb{E}[\|X - \mathbb{E}[X]\|^2] = \mathbb{E}[\|X\|^2] - \|\mathbb{E}[X]\|^2$. Since the squared norm of the expectation is always non-negative, i.e., $\|\mathbb{E}[\hat{g}_{k,i}]\|^2 \ge 0$, the variance is strictly upper-bounded by the second moment:
\begin{equation}
    \mathbb{E}\left[ \|\hat{g}_{k,i} - \mathbb{E}[\hat{g}_{k,i}]\|^2 \right] \le \mathbb{E}[\|\hat{g}_{k,i}\|^2] \le \sigma_{k,i}^2.
\end{equation}
This completes the proof.
\end{IEEEproof}

\begin{remark}
Lemma \ref{lem:zo_properties} establishes the fundamental statistical properties of the Rademacher-based zeroth-order estimator. Notably, its variance bound matches the standard $\mathcal{O}(d)$ dimensional dependence inherent in traditional continuous smoothing techniques (such as Gaussian or uniform spherical sampling). The critical takeaway is that by restricting perturbations to the discrete set $\{-1, 1\}^d$, our estimator preserves the requisite statistical guarantees for convergence, while effectively circumventing the computational overhead associated with continuous random variable generation and non-linear Euclidean normalizations.
\end{remark}

\subsection{Recursive Error Bounds}
With the fundamental statistical properties established, we now investigate the core variance reduction mechanism of our algorithm. The following lemma demonstrates how the momentum update iteratively suppresses the ZO sampling variance, while isolating the deterministic tracking drift.

\begin{lemma}
\label{lem:momentum_error}
Suppose Assumptions \ref{ass:smooth}--\ref{ass:heterogeneity} hold. Let $\mathbf{X}_k = [x_{k,1}, \dots, x_{k,N}]^\top \in \mathbb{R}^{N \times d}$ denote the local variable matrix and define the averaging matrix $J = \frac{1}{N}\mathbf{11}^\top$. We define the network consensus error in its equivalent matrix and summation forms as 
\begin{equation}
	\Xi_{k} \triangleq \frac{1}{N} \|(I-J)\mathbf{X}_k\|_F^2 = \frac{1}{N}\sum_{i=1}^N \|x_{k,i} - \bar{x}_{k}\|^2.
\end{equation}
Based on the global dynamics $\bar{x}_{k+1} = \bar{x}_k - \eta \bar{m}_k$, the global momentum tracking error $E_k \triangleq \bar{m}_k - \nabla F(\bar{x}_k)$ satisfies the following recursive bound:
\begin{align}
    \mathbb{E}[\|E_{k+1}\|^2] &\le \beta \mathbb{E}[\|E_k\|^2] \nonumber \\
    &\quad + \frac{4d(1-\beta)^2}{N} \mathbb{E}\|\nabla F(\bar{x}_{k+1})\|^2 + \mathcal{T}_{drift},
\end{align}
where $\mathcal{T}_{drift}$ explicitly encapsulates the dynamic tracking drift, the expected network consensus error, and the fundamental zeroth-order variance:
\begin{align}
    \mathcal{T}_{drift} &= \frac{2 \beta^2 L^2 \eta^2}{1-\beta} \mathbb{E}[\|\bar{m}_{k}\|^2] \nonumber \\
    &\quad+ 4L^2(1-\beta)\left(1 + \frac{d(1-\beta)}{N}\right) \mathbb{E}[\Xi_{k+1}] \nonumber \\
    &\quad + \frac{(1-\beta)^2}{N}\left( 4d\zeta^2 + \frac{\mu^2 L^2 d^3}{2} \right) + 4(1-\beta)\delta_\mu^2.
\end{align}
\end{lemma}

\begin{IEEEproof}
Let $E_{k+1} \triangleq \bar{m}_{k+1} - \nabla F(\bar{x}_{k+1})$ denote the global momentum tracking error. By averaging the local momentum update rule across all agents, the global momentum sequence evolves as
\begin{equation}
	\bar{m}_{k+1} = \beta \bar{m}_k + (1-\beta)\bar{g}_{k+1},
\end{equation}
where $\bar{g}_{k+1} = \frac{1}{N}\sum_{i=1}^N \hat{g}_{k+1,i}$. To establish a recursive relationship, we expand $E_{k+1}$ by adding and subtracting $\beta \nabla F(\bar{x}_k)$ and $\beta \nabla F(\bar{x}_{k+1})$:
\begin{align}
    E_{k+1} &= \beta(\bar{m}_k - \nabla F(\bar{x}_k)) + \beta(\nabla F(\bar{x}_k) - \nabla F(\bar{x}_{k+1})) \nonumber \\
    &\quad + (1-\beta)(\bar{g}_{k+1} - \nabla F(\bar{x}_{k+1})).
\end{align}
To isolate the pure sampling variance, we decompose the current gradient estimator into its expected value and a zero-mean fluctuation. This naturally splits the tracking error into a deterministic drift component $A_k$ and a stochastic noise component $\xi_{k+1}$:
\begin{align}
    A_k &\triangleq \beta E_k + \beta(\nabla F(\bar{x}_k) - \nabla F(\bar{x}_{k+1})) \nonumber \\
    &\quad + (1-\beta)\left(\frac{1}{N}\sum_{i=1}^N \mathbb{E}_{u}[\hat{g}_{k+1,i}] - \nabla F(\bar{x}_{k+1})\right), \\
    \xi_{k+1} &\triangleq \frac{1-\beta}{N} \sum_{i=1}^N (\hat{g}_{k+1,i} - \mathbb{E}_{u}[\hat{g}_{k+1,i}]).
\end{align}

Conditioned on the algorithm state at iteration $k$, the historical variables are fixed, rendering $A_k$ deterministic. Since the Rademacher perturbation vectors $u_{k+1,i}$ are sampled independently with zero mean, the conditional expectation of the noise term $\xi_{k+1}$ is strictly zero. Consequently, the cross-product vanishes upon taking the total expectation, yielding an orthogonal variance split:
\begin{equation} \label{eq:var_split}
    \mathbb{E}[\|E_{k+1}\|^2] = \mathbb{E}[\|A_k\|^2] + \mathbb{E}[\|\xi_{k+1}\|^2].
\end{equation}

We first evaluate the stochastic noise term $\xi_{k+1}$. Leveraging the independence of sampling across different agents, the cross-covariance terms vanish. Substituting the explicit variance bound from Lemma \ref{lem:zo_properties}, we obtain:
\begin{equation} \label{eq:xi_initial}
    \mathbb{E}[\|\xi_{k+1}\|^2] \le \frac{(1-\beta)^2}{N^2} \sum_{i=1}^N \left( 2d\mathbb{E}\|\nabla f_i(x_{k+1,i})\|^2 + \frac{\mu^2 L^2 d^3}{2} \right).
\end{equation}
To connect the local gradients to the global objective, we apply the standard inequality $\|a+b\|^2 \le 2\|a\|^2 + 2\|b\|^2$ and $L$-smoothness to bound the local deviation:
\begin{equation} \label{eq:local_to_average_grad}
    \|\nabla f_i(x_{k+1,i})\|^2 \le 2L^2\|x_{k+1,i} - \bar{x}_{k+1}\|^2 + 2\|\nabla f_i(\bar{x}_{k+1})\|^2.
\end{equation}
Averaging \eqref{eq:local_to_average_grad} over all $N$ agents, the first term strictly forms the consensus error $2L^2 \Xi_{k+1}$. For the second term, we apply the variance decomposition property $\frac{1}{N}\sum \|z_i\|^2 = \|\bar{z}\|^2 + \frac{1}{N}\sum \|z_i - \bar{z}\|^2$ and invoke the data heterogeneity bound (Assumption \ref{ass:heterogeneity}), yielding:
\begin{align} \label{eq:grad_expansion}
    \frac{1}{N}\sum_{i=1}^N \|\nabla f_i(x_{k+1,i})\|^2 &\le 2L^2 \Xi_{k+1} + 2\|\nabla F(\bar{x}_{k+1})\|^2 \nonumber \\
    &\quad + 2\zeta^2.
\end{align}
Substituting \eqref{eq:grad_expansion} back into \eqref{eq:xi_initial} explicitly isolates the global gradient and noise bounds:
\begin{align} \label{eq:xi_bound_explicit}
    \mathbb{E}[\|\xi_{k+1}\|^2] &\le \frac{4d(1-\beta)^2}{N} \mathbb{E}\|\nabla F(\bar{x}_{k+1})\|^2 \nonumber \\
    &\quad + \frac{4dL^2(1-\beta)^2}{N}\mathbb{E}[\Xi_{k+1}] \nonumber \\
    &\quad + \frac{(1-\beta)^2}{N} \left( 4d\zeta^2 + \frac{\mu^2 L^2 d^3}{2} \right).
\end{align}

Next, we bound the deterministic drift $A_k$. To enforce a strict contraction on the historical error $E_k$, we apply the Peter-Paul inequality, $\|a+b\|^2 \le (1+\alpha)\|a\|^2 + (1+1/\alpha)\|b\|^2$, by choosing $\alpha = \frac{1-\beta}{\beta}$. This strictly ensures the coefficient of $\|E_k\|^2$ is $(1+\alpha)\beta^2 = \beta$, while the residual term is penalized by $(1+1/\alpha) = \frac{1}{1-\beta}$:
\begin{align} \label{eq:Ak_peter_paul}
    \|A_k\|^2 &\le \beta \|E_k\|^2 + \frac{1}{1-\beta} \Big\| \beta(\nabla F(\bar{x}_k) - \nabla F(\bar{x}_{k+1})) \nonumber \\
    &\quad + (1-\beta)\left(\frac{1}{N}\sum_{i=1}^N \mathbb{E}_{u}[\hat{g}_{k+1,i}] - \nabla F(\bar{x}_{k+1})\right) \Big\|^2.
\end{align}
Applying $\|x+y\|^2 \le 2\|x\|^2 + 2\|y\|^2$ to the squared residual, we bound the two components separately. The gradient tracking difference is bounded by $L$-smoothness and the update rule:
\begin{equation} \label{eq:ak_part1}
    2\beta^2 \|\nabla F(\bar{x}_k) - \nabla F(\bar{x}_{k+1})\|^2 \le 2\beta^2 L^2 \eta^2 \|\bar{m}_k\|^2.
\end{equation}
For the expected zeroth-order deviation, we explicitly split it into the fundamental smoothing bias and the spatial consensus deviation:
\begin{align} \label{eq:ak_part2}
    &2(1-\beta)^2 \left\| \frac{1}{N}\sum_{i=1}^N (\mathbb{E}_{u}[\hat{g}_{k+1,i}] - \nabla f_i(\bar{x}_{k+1})) \right\|^2 \nonumber \\
    &\le 4(1-\beta)^2 \frac{1}{N}\sum_{i=1}^N \|\mathbb{E}_{u}[\hat{g}_{k+1,i}] - \nabla f_i(x_{k+1,i})\|^2 \nonumber \\
    &\quad + 4(1-\beta)^2 \frac{1}{N}\sum_{i=1}^N \|\nabla f_i(x_{k+1,i}) - \nabla f_i(\bar{x}_{k+1})\|^2 \nonumber \\
    &\le 4(1-\beta)^2 \delta_\mu^2 + 4(1-\beta)^2 L^2 \Xi_{k+1},
\end{align}
where we utilized the bias bound $\delta_\mu$ derived in Lemma \ref{lem:zo_properties} and applied Jensen's inequality over the summation. Substituting \eqref{eq:ak_part1} and \eqref{eq:ak_part2} back into \eqref{eq:Ak_peter_paul}, and scaling by the penalty coefficient $\frac{1}{1-\beta}$, we obtain the complete bound for the drift component:
\begin{align} \label{eq:Ak_bound_final}
    \mathbb{E}[\|A_k\|^2] &\le \beta \mathbb{E}[\|E_k\|^2] + \frac{2\beta^2 L^2 \eta^2}{1-\beta} \mathbb{E}[\|\bar{m}_k\|^2] \nonumber \\
    &\quad + 4(1-\beta)\left(\delta_\mu^2 + L^2 \mathbb{E}[\Xi_{k+1}]\right).
\end{align}

Finally, substituting the explicit variance bound \eqref{eq:xi_bound_explicit} and the deterministic drift bound \eqref{eq:Ak_bound_final} into the orthogonal split \eqref{eq:var_split}, and collecting the coefficients for the network consensus error $\mathbb{E}[\Xi_{k+1}]$ and the constant variance terms into $\mathcal{T}_{drift}$, precisely establishes the stated recursive inequality. This completes the proof.
\end{IEEEproof}

\begin{remark}
Lemma \ref{lem:momentum_error} formalizes the core variance reduction mechanism of ZO-MGT. The orthogonal variance split \eqref{eq:var_split} and the Peter-Paul inequality \eqref{eq:Ak_peter_paul} ensure a strict geometric contraction, that is controlled by $\beta$, on the historical tracking error. Furthermore, examining the residual drift $\mathcal{T}_{drift}$ reveals that both the zeroth-order sampling noise and the data heterogeneity bound $\zeta^2$ are suppressed by a factor of $\mathcal{O}((1-\beta)^2 / N)$. This theoretically supports our design intuition: a large momentum factor ($\beta \to 1$) combined with network aggregation acts as a low-pass filter. It effectively prevents the severe variance explosion that typically occurs when standard gradient tracking is applied to noisy derivative-free estimators.
\end{remark}

The momentum tracking error derived above is tightly coupled with the network consensus error $\Xi_k$. Due to the decentralized nature of the network, extreme data heterogeneity $\zeta^2$ constantly drives the agents away from the global average. To decouple this interaction and guarantee network agreement, we establish a cumulative bound for the consensus error over $T$ iterations.

\begin{lemma}
\label{lem:consensus_recursion}
Provided that Assumptions \ref{ass:smooth}--\ref{ass:heterogeneity} hold. The consensus error $\Xi_k \triangleq \frac{1}{N} \|(I-J)\mathbf{X}_k\|_F^2$ satisfies the following expected cumulative bound over $T$ iterations:
\begin{align} \label{eq:finite_sum_xi}
    \sum_{k=0}^{T-1} \mathbb{E}[\Xi_k] &\le \frac{C_0}{(1-\rho)^3} \mathbb{E}[\Xi_0] + \frac{C_1 \eta^2}{(1-\rho)^4} \sum_{k=0}^{T-1} \mathbb{E}\|\nabla F(\bar{x}_k)\|^2 \nonumber \\
    &\quad + \frac{C_2 T \eta^2}{(1-\rho)^4} \sigma_{ZO}^2 + \frac{C_3 \eta^2}{(1-\rho)^4} \frac{1}{N} \mathbb{E}\|\mathbf{G}_0\|_F^2,
\end{align}
where $\sigma_{ZO}^2 \triangleq 4d\zeta^2 + \frac{1}{2}\mu^2 L^2 d^3$ represents the fundamental zeroth-order variance floor. The explicit constants $C_0, C_1, C_2, C_3$ depend strictly on the momentum parameter $\beta$ and dimension $d$, and are independent of $T$ and $\eta$.
\end{lemma}

\begin{IEEEproof}
To establish the cumulative bound, we analyze the coupled dynamics of the expected consensus error $\mathbb{E}[\Xi_k]$ and the gradient tracking error $\Upsilon_k \triangleq \frac{1}{N}\mathbb{E}\|\mathbf{Y}_k - J\mathbf{Y}_k\|_F^2$. Recalling the contraction property of the doubly stochastic mixing matrix, i.e., $\rho \triangleq \|W-J\|_2 < 1$ and using Young's inequality, the standard decentralized dynamics yield the following single-step inequalities:
\begin{align}
    \mathbb{E}[\Xi_{k+1}] &\le \frac{1+\rho^2}{2} \mathbb{E}[\Xi_k] + \frac{2\eta^2}{1-\rho^2} \Upsilon_k, \label{eq:xi_dyn} \\
    \Upsilon_{k+1} &\le \frac{1+\rho^2}{2} \Upsilon_k + \frac{2}{1-\rho^2} \frac{1}{N} \mathbb{E}\|\Delta \mathbf{M}_{k+1}\|_F^2, \label{eq:ups_dyn}
\end{align}
where $\Delta \mathbf{M}_{k+1} \triangleq \mathbf{M}_{k+1} - \mathbf{M}_k$. To explicitly decouple these errors, we adopt a summation-based approach. Summing \eqref{eq:xi_dyn} and \eqref{eq:ups_dyn} from $k=0$ to $T-1$, and utilizing the non-negativity $\mathbb{E}[\Xi_T] \ge 0$ and $\Upsilon_T \ge 0$, we obtain:
\begin{align}
    \frac{1-\rho^2}{2} \sum_{k=0}^{T-1} \mathbb{E}[\Xi_k] &\le \mathbb{E}[\Xi_0] + \frac{2\eta^2}{1-\rho^2} \sum_{k=0}^{T-1} \Upsilon_k, \label{eq:sum_xi_intermediate} \\
    \frac{1-\rho^2}{2} \sum_{k=0}^{T-1} \Upsilon_k &\le \Upsilon_0 + \frac{2}{1-\rho^2} \sum_{k=0}^{T-1} \frac{1}{N} \mathbb{E}\|\Delta \mathbf{M}_{k+1}\|_F^2. \label{eq:sum_ups_intermediate}
\end{align}

Multiplying \eqref{eq:sum_ups_intermediate} by $\frac{4}{(1-\rho^2)^2}$, we can isolate $\frac{2\eta^2}{1-\rho^2}\sum_{k=0}^{T-1} \Upsilon_k$ and directly substitute it into \eqref{eq:sum_xi_intermediate}. Noting that the exact initialization $\mathbf{Y}_0 = \mathbf{G}_0$ implies $\Upsilon_0 \le \frac{1}{N}\mathbb{E}\|\mathbf{G}_0\|_F^2$, we multiply the resulting inequality by $\frac{2}{1-\rho^2}$ to yield the decoupled consensus bound:
\begin{align} \label{eq:sum_xi_coupled}
    \sum_{k=0}^{T-1} \mathbb{E}[\Xi_k] &\le \frac{2}{1-\rho^2} \mathbb{E}[\Xi_0] + \frac{8\eta^2}{(1-\rho^2)^3} \frac{1}{N} \mathbb{E}\|\mathbf{G}_0\|_F^2 \nonumber \\
    &\quad + \frac{16\eta^2}{(1-\rho^2)^4} \sum_{k=0}^{T-1} \frac{1}{N} \mathbb{E}\|\Delta \mathbf{M}_{k+1}\|_F^2.
\end{align}

To bound the cumulative momentum variation, we apply the inequality $\|a-b\|^2 \le 2\|a\|^2 + 2\|b\|^2$ to the update rule $\mathbf{M}_{k+1} = \beta \mathbf{M}_k + (1-\beta) \mathbf{G}_{k+1}$, which provides the single-step bound:
\begin{equation} \label{eq:single_delta_M}
    \frac{1}{N}\mathbb{E}\|\Delta \mathbf{M}_{k+1}\|_F^2 \le 2(1-\beta)^2 \left( \frac{1}{N}\mathbb{E}\|\mathbf{G}_{k+1}\|_F^2 + \frac{1}{N}\mathbb{E}\|\mathbf{M}_k\|_F^2 \right).
\end{equation}
By unrolling the momentum update, $\mathbf{M}_k$ can be expressed as a convex combination of historical gradients: $\mathbf{M}_k = \beta^k \mathbf{G}_0 + \sum_{t=1}^k (1-\beta)\beta^{k-t} \mathbf{G}_t$. Applying Jensen's inequality to the squared Frobenius norm, we have:
\begin{equation} \label{eq:jensen_M}
    \|\mathbf{M}_k\|_F^2 \le \beta^k \|\mathbf{G}_0\|_F^2 + \sum_{t=1}^k (1-\beta)\beta^{k-t} \|\mathbf{G}_t\|_F^2.
\end{equation}
Summing \eqref{eq:jensen_M} over $T$ iterations and interchanging the order of summation, the cumulative momentum error is strictly bounded by the raw gradient error:
\begin{align} \label{eq:sum_M_bound}
    \sum_{k=0}^{T-1} \frac{1}{N}\mathbb{E}\|\mathbf{M}_k\|_F^2 &\le \sum_{k=0}^{T-1} \beta^k \frac{1}{N}\mathbb{E}\|\mathbf{G}_0\|_F^2 \nonumber \\
    &\quad + \sum_{t=1}^{T-1} \frac{1}{N}\mathbb{E}\|\mathbf{G}_t\|_F^2 \sum_{k=t}^{T-1} (1-\beta)\beta^{k-t} \nonumber \\
    &\le \frac{1}{1-\beta} \frac{1}{N}\mathbb{E}\|\mathbf{G}_0\|_F^2 + \sum_{k=0}^{T-1} \frac{1}{N}\mathbb{E}\|\mathbf{G}_k\|_F^2.
\end{align}
Substituting \eqref{eq:sum_M_bound} back into the sum of \eqref{eq:single_delta_M}, the cumulative momentum difference is simplified to:
\begin{align} \label{eq:delta_M_sum_final}
    \sum_{k=0}^{T-1} \frac{1}{N}\mathbb{E}\|\Delta \mathbf{M}_{k+1}\|_F^2 &\le 4(1-\beta)^2 \sum_{k=0}^{T-1} \frac{1}{N}\mathbb{E}\|\mathbf{G}_k\|_F^2 \nonumber \\
    &\quad + 2(1-\beta) \frac{1}{N}\mathbb{E}\|\mathbf{G}_0\|_F^2.
\end{align}

Furthermore, to connect the raw gradient $\mathbf{G}_k$ with the global objective, we apply the explicit variance bound from Lemma \ref{lem:zo_properties}. Decomposing the local gradient via $L$-smoothness and invoking Assumption \ref{ass:heterogeneity}, we explicitly obtain:
\begin{equation} \label{eq:G_bound}
    \frac{1}{N}\mathbb{E}\|\mathbf{G}_k\|_F^2 \le 4dL^2 \mathbb{E}[\Xi_k] + 4d\mathbb{E}\|\nabla F(\bar{x}_k)\|^2 + \sigma_{ZO}^2.
\end{equation}

Finally, we substitute the gradient expansion \eqref{eq:G_bound} into \eqref{eq:delta_M_sum_final}, and then substitute the resulting bound into the decoupled consensus inequality \eqref{eq:sum_xi_coupled}. Collecting all identical terms yields the complete self-feedback inequality:
\begin{align} \label{eq:xi_contraction_prep}
    \sum_{k=0}^{T-1} \mathbb{E}[\Xi_k] &\le \frac{256 d L^2 (1-\beta)^2 \eta^2}{(1-\rho^2)^4} \sum_{k=0}^{T-1} \mathbb{E}[\Xi_k] \nonumber \\
    &\quad + \frac{256 d (1-\beta)^2 \eta^2}{(1-\rho^2)^4} \sum_{k=0}^{T-1} \mathbb{E}\|\nabla F(\bar{x}_k)\|^2 \nonumber \\
    &\quad + \frac{64(1-\beta)^2 T \eta^2}{(1-\rho)^4} \sigma_{ZO}^2 \nonumber \\
    &\quad + \left( \frac{8}{(1-\rho^2)^3} + \frac{32(1-\beta)}{(1-\rho^2)^4} \right) \eta^2 \frac{1}{N}\mathbb{E}\|\mathbf{G}_0\|_F^2 \nonumber \\
    &\quad + \frac{2}{1-\rho^2} \mathbb{E}[\Xi_0].
\end{align}

To guarantee a valid convergence bound, we require the coefficient of the feedback term $\sum \mathbb{E}[\Xi_k]$ on the right-hand side to be bounded by $1/2$. Under the condition $256 d L^2 (1-\beta)^2 \eta^2 \le \frac{1}{2} (1-\rho^2)^4$, we rearrange the inequality by moving the feedback term to the left-hand side, resulting in $(1 - 1/2) \sum_{k=0}^{T-1} \mathbb{E}[\Xi_k]$. Multiplying the entire inequality by 2, and utilizing the topological property $(1-\rho^2)^4 \ge (1-\rho)^4 \ge (1-\rho)^3$, we cleanly extract the structural constants: $C_0 = 4$, $C_1 = 512 d (1-\beta)^2$, $C_2 = 128 (1-\beta)^2$, and $C_3 = 16 + 64(1-\beta)$. This precisely yields the stated bound in \eqref{eq:finite_sum_xi} and completes the proof.
\end{IEEEproof}

\subsection{Global Convergence Rate}
Having bounded both the internal momentum tracking error and the decentralized network consensus error, we proceed to evaluate the descent behavior of the objective function along the global average trajectory.

\begin{lemma}
\label{lem:descent_lemma}
Suppose Assumptions \ref{ass:smooth}--\ref{ass:heterogeneity} hold. If the step size fulfills $\eta \le \frac{1}{2L}$, the global objective function satisfies the following descent inequality:
\begin{align} \label{eq:descent_inequality}
    F(\bar{x}_{k+1}) &\le F(\bar{x}_k) - \frac{\eta}{2} \|\nabla F(\bar{x}_k)\|^2 \nonumber \\
    &\quad + \frac{\eta}{2} \|\bar{m}_k - \nabla F(\bar{x}_k)\|^2 - \frac{\eta}{4} \|\bar{m}_k\|^2.
\end{align}
\end{lemma}

\begin{IEEEproof}
By the $L$-smoothness of the global objective function $F$ (Assumption \ref{ass:smooth}), for any two points $x$ and $y$, we have the standard quadratic upper bound:
\begin{equation}
    F(y) \le F(x) + \langle \nabla F(x), y - x \rangle + \frac{L}{2} \|y - x\|^2.
\end{equation}
Setting $y = \bar{x}_{k+1}$ and $x = \bar{x}_k$, and substituting the global average update dynamics $\bar{x}_{k+1} - \bar{x}_k = -\eta \bar{m}_k$ established in Lemma \ref{lem:global_dynamics}, we obtain:
\begin{equation} \label{eq:smoothness_sub}
    F(\bar{x}_{k+1}) \le F(\bar{x}_k) - \eta \langle \nabla F(\bar{x}_k), \bar{m}_k \rangle + \frac{L\eta^2}{2} \|\bar{m}_k\|^2.
\end{equation}

To handle the inner product term, we utilize the fundamental algebraic identity $2\langle a, b \rangle = \|a\|^2 + \|b\|^2 - \|a - b\|^2$. By setting $a = \nabla F(\bar{x}_k)$ and $b = \bar{m}_k$, and multiplying both sides by $-\frac{\eta}{2}$, we can explicitly expand the cross term as:
\begin{align} \label{eq:inner_product_expansion}
    -\eta \langle \nabla F(\bar{x}_k), \bar{m}_k \rangle &= -\frac{\eta}{2} \|\nabla F(\bar{x}_k)\|^2 - \frac{\eta}{2} \|\bar{m}_k\|^2 \nonumber \\
    &\quad + \frac{\eta}{2} \|\bar{m}_k - \nabla F(\bar{x}_k)\|^2.
\end{align}

Next, we substitute \eqref{eq:inner_product_expansion} back into the smoothness bound \eqref{eq:smoothness_sub}. Grouping the terms associated with $\|\bar{m}_k\|^2$ together yields:
\begin{align} \label{eq:descent_grouped}
    F(\bar{x}_{k+1}) &\le F(\bar{x}_k) - \frac{\eta}{2} \|\nabla F(\bar{x}_k)\|^2 + \frac{\eta}{2} \|\bar{m}_k - \nabla F(\bar{x}_k)\|^2 \nonumber \\
    &\quad - \left( \frac{\eta}{2} - \frac{L\eta^2}{2} \right) \|\bar{m}_k\|^2.
\end{align}
We can factor out $\frac{\eta}{2}$ from the final term to rewrite its precise coefficient as $-\frac{\eta}{2}(1 - L\eta)$. 

Finally, we apply the step size restriction. Since the constant step size is strictly bounded by $\eta \le \frac{1}{2L}$, it naturally follows that $L\eta \le \frac{1}{2}$. Consequently, we have $1 - L\eta \ge \frac{1}{2}$. Applying this inequality to the momentum norm coefficient, we establish the strict upper bound:
\begin{equation}
    - \frac{\eta}{2}(1 - L\eta) \|\bar{m}_k\|^2 \le - \frac{\eta}{2} \left( \frac{1}{2} \right) \|\bar{m}_k\|^2 = - \frac{\eta}{4} \|\bar{m}_k\|^2.
\end{equation}
Substituting this final bound into \eqref{eq:descent_grouped} directly yields the stated descent inequality in \eqref{eq:descent_inequality}, which completes the proof.
\end{IEEEproof}

Finally, we consolidate the momentum tracking bound, the cumulative consensus bound, and the global descent inequality into a unified Lyapunov function. This allows us to establish the main convergence theorem and extract the exact steady-state error floor.

\begin{theorem}
\label{thm:convergence}
Let Assumptions \ref{ass:smooth}--\ref{ass:heterogeneity} hold, and assume the global objective is lower bounded such that $\inf_{x} F(x) \ge F^\star > -\infty$. Suppose the momentum factor $\beta \in (0,1)$ is chosen such that $1-\beta \le \frac{N}{16d}$. If the constant step size $\eta$ satisfies the following condition:
\begin{equation} \label{eq:step_size_condition}
    \eta \le \min \left\{ \frac{1}{2L}, \frac{1-\beta}{\sqrt{8}\beta L}, \frac{(1-\rho)^2}{L \sqrt{8 C_\Xi C_1}} \right\},
\end{equation}
where $C_\Xi$ and $C_1$ are positive structural constants independent of $T$ and $\eta$, then the sequence $\{\bar{x}_k\}_{k=0}^{T-1}$ generated by Algorithm \ref{alg:zo_mgt} satisfies:
\begin{align} \label{eq:theorem_result}
    \frac{1}{T} \sum_{k=0}^{T-1} \mathbb{E}\|\nabla F(\bar{x}_k)\|^2 &\le \frac{8(\tilde{V}_0 - F^\star)}{\eta T} \nonumber \\
    &\quad + \mathcal{O}\left( \mu^2 L^2 d^3 + \frac{(1-\beta)^2 d}{N}\zeta^2 \right),
\end{align}
where $\tilde{V}_0 \triangleq F(\bar{x}_0) + \frac{\eta}{1-\beta} \mathbb{E}\|\bar{m}_0 - \nabla F(\bar{x}_0)\|^2 + C_{init} \eta \mathbb{E}\|\mathbf{G}_0\|_F^2$ defines the exact initial augmented energy, and the constant $C_{init} > 0$ depends strictly on the network topology and the momentum parameter. 
\end{theorem}

\begin{IEEEproof}
To establish the global convergence rate, we construct a composite Lyapunov sequence $V_k \triangleq \mathbb{E}[F(\bar{x}_k)] + \lambda \mathbb{E}[\|E_k\|^2]$, which tightly couples the objective descent with the expected momentum tracking error $\mathbb{E}[\|E_k\|^2]$, where $E_k \triangleq \bar{m}_k - \nabla F(\bar{x}_k)$. By defining the coupling parameter specifically as $\lambda \triangleq \frac{\eta}{1-\beta}$, we evaluate the single-step Lyapunov drift $V_{k+1} - V_k$.

Substituting the objective descent bound from Lemma \ref{lem:descent_lemma} and multiplying the recursive momentum bound from Lemma \ref{lem:momentum_error} by $\lambda$, we obtain the explicitly expanded drift:
\begin{align} \label{eq:lyap_expansion}
    V_{k+1} - V_k &\le -\frac{\eta}{2} \mathbb{E}\|\nabla F(\bar{x}_k)\|^2 + \frac{\eta}{2} \mathbb{E}[\|E_k\|^2] - \frac{\eta}{4} \mathbb{E}\|\bar{m}_k\|^2 \nonumber \\
    &\quad + \lambda (\beta - 1) \mathbb{E}[\|E_k\|^2] \nonumber \\ 
    &+ \lambda \frac{4d(1-\beta)^2}{N} \mathbb{E}\|\nabla F(\bar{x}_{k+1})\|^2 \nonumber \\
    &\quad + \lambda \frac{2\beta^2 L^2 \eta^2}{1-\beta} \mathbb{E}\|\bar{m}_k\|^2 + \lambda \cdot \mathcal{T}_{error},
\end{align}
where $\mathcal{T}_{error}$ encapsulates the remaining network consensus error and the zeroth-order variance floor derived in Lemma \ref{lem:momentum_error}. 

Crucially, the specific choice of $\lambda$ yields $\lambda(\beta - 1) \mathbb{E}[\|E_k\|^2] = -\eta \mathbb{E}[\|E_k\|^2]$. When combined with the $\frac{\eta}{2} \mathbb{E}[\|E_k\|^2]$ term from the descent lemma, this generates a strictly negative tracking drift $-\frac{\eta}{2} \mathbb{E}[\|E_k\|^2] \le 0$, completely eliminating the momentum error $\mathbb{E}[\|E_k\|^2]$ from the upper bound. Collecting the identical terms (and aligning the gradient indices via summation over $T$ iterations), the single-step dynamic simplifies to:
\begin{align} \label{eq:lyap_drift_exact}
    V_{k+1} - V_k &\le - \frac{\eta}{2} \left( 1 - \frac{8d(1-\beta)}{N} \right) \mathbb{E}\|\nabla F(\bar{x}_k)\|^2 \nonumber \\
    &\quad - \eta \left( \frac{1}{4} - \frac{2\beta^2 L^2 \eta^2}{(1-\beta)^2} \right) \mathbb{E}\|\bar{m}_k\|^2 \nonumber \\
    &\quad + C_\Xi \eta L^2 \mathbb{E}[\Xi_{k+1}] + \eta C_{ZO},
\end{align}
where we define the network coefficient $C_\Xi \triangleq 4(1-\beta)(1+\frac{d(1-\beta)}{N})$ and the fundamental variance bound $C_{ZO} \triangleq \frac{(1-\beta)^2}{N}(4d\zeta^2 + \frac{\mu^2 L^2 d^3}{2}) + 4(1-\beta)\delta_\mu^2$.

To guarantee strict objective descent, we systematically control the coefficients in \eqref{eq:lyap_drift_exact}. The momentum condition $1-\beta \le \frac{N}{16d}$ mathematically ensures $\frac{8d(1-\beta)}{N} \le \frac{1}{2}$, thereby bounding the leading coefficient of the gradient norm by $-\frac{\eta}{4}$. Concurrently, the step size condition $\eta \le \frac{1-\beta}{\sqrt{8}\beta L}$ directly guarantees $\frac{1}{4} - \frac{2\beta^2 L^2 \eta^2}{(1-\beta)^2} \ge 0$, which structurally removes the auxiliary variable $\mathbb{E}\|\bar{m}_k\|^2$ from the bound.

Summing the simplified drift from $k=0$ to $T-1$ produces the following telescoping sum:
\begin{align} \label{eq:summed_drift_initial}
    \frac{\eta}{4} \sum_{k=0}^{T-1} \mathbb{E}\|\nabla F(\bar{x}_k)\|^2 &\le V_0 - \mathbb{E}[F(\bar{x}_T)] + \eta T C_{ZO} \nonumber \\
    &\quad + C_\Xi \eta L^2 \sum_{k=0}^{T-1} \mathbb{E}[\Xi_k].
\end{align}
Because the cumulative consensus error contains nested historical gradients, we explicitly substitute the bound derived in Lemma \ref{lem:consensus_recursion} (Eq. \eqref{eq:finite_sum_xi}) into \eqref{eq:summed_drift_initial}. Moving all the gradient sequence terms to the left-hand side groups them as follows:
\begin{align} \label{eq:grouped_gradient}
    &\left( \frac{\eta}{4} - \frac{C_\Xi C_1 \eta^3 L^2}{(1-\rho)^4} \right) \sum_{k=0}^{T-1} \mathbb{E}\|\nabla F(\bar{x}_k)\|^2 \nonumber \\
    &\le V_0 - \mathbb{E}[F(\bar{x}_T)] + \eta T \left( C_{ZO} + \frac{C_\Xi C_2 \eta^2 L^2}{(1-\rho)^4} \sigma_{ZO}^2 \right) \nonumber \\
    &\quad + \frac{C_\Xi C_3 \eta^3 L^2}{(1-\rho)^4} \frac{1}{N}\mathbb{E}\|\mathbf{G}_0\|_F^2 + \frac{C_\Xi C_0 \eta L^2}{(1-\rho)^3} \mathbb{E}[\Xi_0].
\end{align}

To establish a strict global contraction, the network-induced gradient penalty must be universally dominated by the primary descent term. We mathematically enforce this by requiring the coefficient to be at least $\frac{\eta}{8}$:
\begin{equation} \label{eq:contraction_enforcement}
    \frac{\eta}{4} - \frac{C_\Xi C_1 \eta^3 L^2}{(1-\rho)^4} \ge \frac{\eta}{8} \implies \frac{C_\Xi C_1 \eta^2 L^2}{(1-\rho)^4} \le \frac{1}{8}.
\end{equation}
Solving this precise inequality yields the final structural constraint $\eta \le \frac{(1-\rho)^2}{L \sqrt{8 C_\Xi C_1}}$, exactly matching the step size condition \eqref{eq:step_size_condition}. 

Under this rigorous step size regime, the leading coefficient simplifies to $\frac{\eta}{8}$. We formally absorb the static initialization terms ($\mathbb{E}[\Xi_0]$ and $\mathbf{G}_0$) into the augmented initial energy $\tilde{V}_0$. Dividing both sides of the inequality by $\frac{\eta T}{8}$ and invoking the objective lower bound $\mathbb{E}[F(\bar{x}_T)] \ge F^\star$ isolates the gradient norm. Finally, extracting the dominant steady-state components with respect to $d$, $\zeta^2$, and $\mu$ using the $\mathcal{O}(\cdot)$ notation yields the ultimate convergence result \eqref{eq:theorem_result}. This completes the proof.
\end{IEEEproof}

\begin{remark}
Equation \eqref{eq:theorem_result} reveals that under a constant step size $\eta$, ZO-MGT converges to a stationary neighborhood of size $\mathcal{O}(\mu^2 L^2 d^3 + \frac{(1-\beta)^2}{N} d \zeta^2)$. This explicit formulation highlights two core advantages of our architecture: (1) the network scale $N$ acts as an inherent variance reducer for data heterogeneity, demonstrating the collaborative benefit of distributed sampling; and (2) the steady-state error induced by extreme non-IID data $\zeta^2$ scales with the square of the residual momentum coefficient, i.e., $\mathcal{O}((1-\beta)^2)$. By selecting a large momentum factor $\beta \to 1$, the detrimental impact of data heterogeneity is aggressively suppressed, enabling much tighter consensus and significantly lower error floors than traditional zero-order baselines.
\end{remark}

\section{Numerical Experiments}
\label{sec:experiments}

To empirically validate the theoretical convergence properties and the variance reduction capability of the proposed ZO-MGT algorithm, we evaluate its performance on a non-convex classification task.

\subsection{Experimental Setup and Baselines}

We simulate a decentralized binary classification scenario using the widely adopted \texttt{a9a} dataset from the LIBSVM library. The feature dimension is $d = 124$ (including the bias term), and the total number of samples is $N_{total} = 32561$. The network consists of $N=20$ agents. 

To strictly satisfy the non-convexity and smoothness conditions delineated in Assumption \ref{ass:smooth}, the local cost function $f_i(x)$ for each agent $i$ is formulated as a non-linear least squares problem. This formulation combines the mean squared error (MSE) with a sigmoid activation function, supplemented by an $\ell_2$-regularization term. Specifically, the local objective is given by:
\begin{equation}
    f_i(x) = \frac{1}{|D_i|} \sum_{j \in D_i} \left( \frac{1}{1 + \exp(-a_j^\top x)} - y_j \right)^2 + \frac{\lambda}{2} \|x\|^2,
\end{equation}
where $D_i$ denotes the local dataset exclusively assigned to agent $i$, $a_j \in \mathbb{R}^d$ is the feature vector, $y_j \in \{0, 1\}$ is the corresponding class label, and $\lambda = 0.001$ is the regularization parameter. 

To rigorously evaluate the algorithmic resilience under severe data heterogeneity (Assumption \ref{ass:heterogeneity}), we adopt a pathological data partitioning strategy. Specifically, all data samples are globally sorted by their labels prior to being partitioned into $N$ equal-sized, disjoint subsets. Consequently, the local dataset $D_i$ for the vast majority of agents contains samples from strictly a single class. This extreme data heterogeneity induces a significantly large local gradient variance bound $\zeta^2$.

The communication topology is modeled as an Erdős-Rényi random graph $\mathcal{G}=(\mathcal{V}, \mathcal{E})$ with a connectivity probability of $p=0.3$, simulating an unstructured peer-to-peer network. To guarantee the doubly stochastic property and strictly positive spectral gap required by Assumption \ref{ass:graph}, the mixing matrix $W = [w_{ij}]$ is constructed using the Metropolis-Hastings rule:
\begin{equation}
    w_{ij} = 
    \begin{cases}
        \frac{1}{\max\{d_i, d_j\} + 1}, & \text{if } \{i,j\} \in \mathcal{E}, \\
        1 - \sum_{k \in \mathcal{N}_i} w_{ik}, & \text{if } i = j, \\
        0, & \text{otherwise,}
    \end{cases}
\end{equation}
where $d_i$ denotes the degree of agent $i$, and $\mathcal{N}_i$ is its neighborhood set. The resulting spectral gap parameter for the generated graph is explicitly calculated as $\rho = \|W - \frac{1}{N}\mathbf{11}^\top\|_2 \approx 0.7592$.

We comprehensively compare our ZO-MGT algorithm against two well-established distributed zeroth-order baselines:
\begin{itemize}
    \item Tang Alg 1  \cite{tangDistributedZeroorderAlgorithms2021} : A standard distributed gradient descent method employing a 2-point central-difference estimator over uniformly sampled spherical directions.
    \item Tang Alg 2  \cite{tangDistributedZeroorderAlgorithms2021} : A gradient tracking algorithm utilizing a deterministic, full $2d$-point coordinate-descent estimator.
\end{itemize}

It is worth noting that while several recent works have explored variance reduction in zeroth-order settings, e.g., incorporating communication compression or centralized server aggregation, they operate under fundamentally different network assumptions, such as star topologies or unconstrained query budgets. Thus, a direct empirical comparison is omitted to ensure a fair evaluation of the core tracking and momentum mechanisms in a fully decentralized, strictly two-point query regime.

For an equitable comparison, all algorithms operate under a constant step size $\eta = 0.05$ and a perturbation smoothing radius $\mu = 0.01$. The momentum factor for ZO-MGT is set to $\beta = 0.9$. To quantitatively assess the performance, we track two primary metrics: the squared global gradient norm $\|\nabla F(\bar{x}_k)\|^2$ (measuring stationarity) and the network consensus error $\frac{1}{N}\sum_{i=1}^N \|x_{k,i} - \bar{x}_k\|^2$ (measuring network agreement). Furthermore, to simulate realistic black-box deployments where objective state monitoring is required, the execution time strictly incorporates all necessary function evaluations.

\subsection{Results and Discussion}
Due to the prohibitive computational complexity of Tang Alg 2, we terminate its execution early at 40 iterations, whereas ZO-MGT and Tang Alg 1 are evaluated over the full 1000 iterations.

\begin{figure}[h]
    \centering
    \subfloat[Global Gradient Norm vs Iterations]{\includegraphics[width=0.9\linewidth]{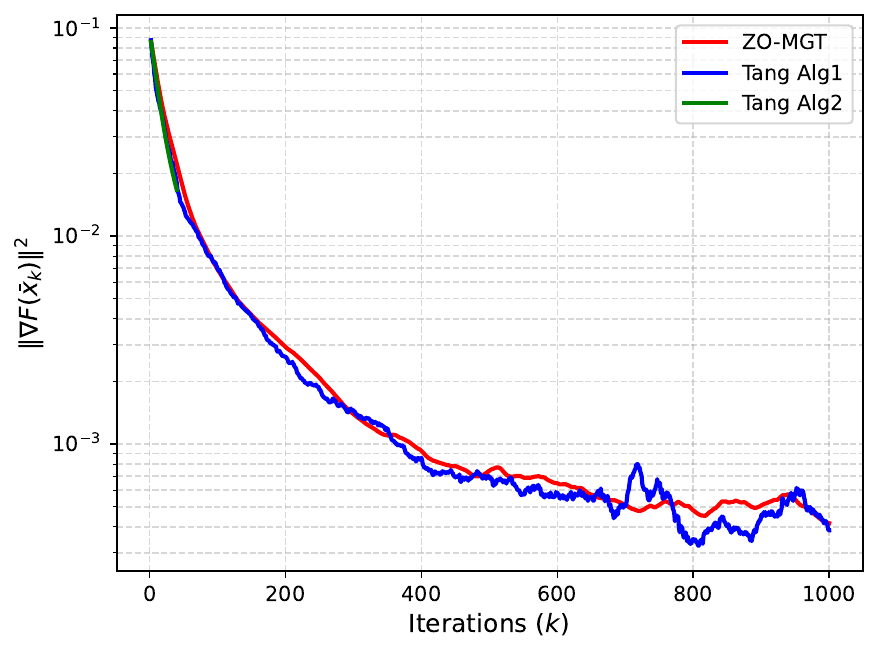}\label{fig:grad_iter}}
    \vspace{0.5em} 
    \\
    \subfloat[Consensus Error vs Iterations]{\includegraphics[width=0.9\linewidth]{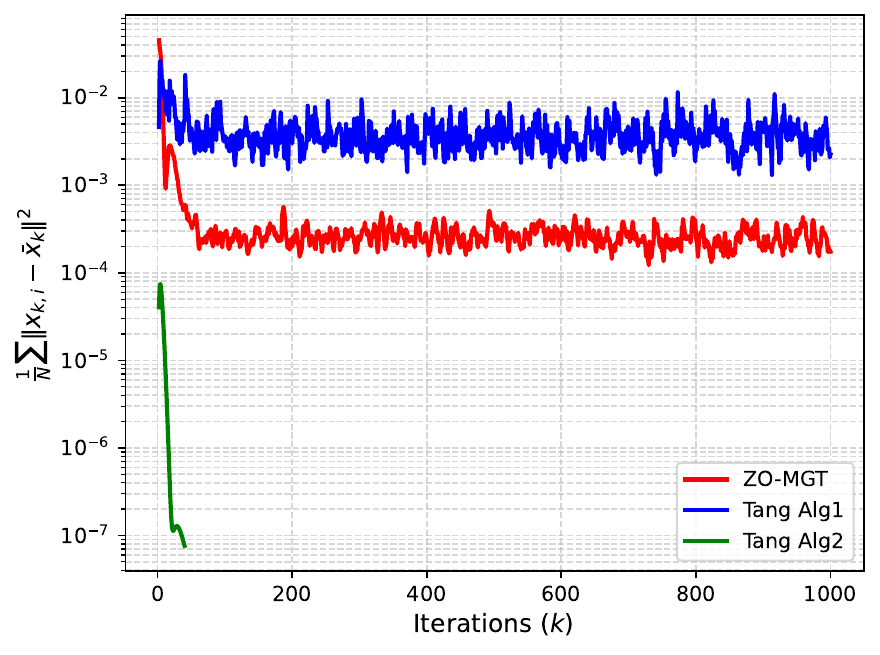}\label{fig:cons_iter}}
    \caption{Convergence evaluation with respect to iterations. }
    \label{fig:exp_iters}
\end{figure}

\subsubsection{Variance Reduction and Consensus Stability }
As illustrated in Fig. \ref{fig:exp_iters}(a), ZO-MGT and Tang Alg 1 exhibit comparable initial descent trajectories in the global gradient norm, confirming that the Rademacher-based gradient estimators employed in our framework maintain approximation accuracy similar to traditional spherical perturbations. However, their asymptotic consensus behaviors diverge significantly, as shown in Fig. \ref{fig:exp_iters}(b). Tang Alg 1 exhibits a higher steady-state error floor ($\approx 10^{-2}$), as it lacks a dedicated mechanism to suppress the compounded variance induced by the coupling of severe data heterogeneity and zeroth-order sampling noise. 

In contrast, ZO-MGT effectively mitigates this variance through its momentum-based update. By filtering the sampling fluctuations, it achieves a consensus error ($\approx 10^{-4}$) that is approximately two orders of magnitude lower than Tang Alg 1, thereby validating our theoretical results in Theorem \ref{thm:convergence}. While Tang Alg 2 maintains the lowest error floor due to its deterministic full-coordinate gradient evaluation, ZO-MGT provides a much more favorable trade-off between tracking precision and per-iteration query complexity.

\begin{figure}[h]
    \centering
    \subfloat[Global Gradient Norm vs Time]{\includegraphics[width=0.9\linewidth]{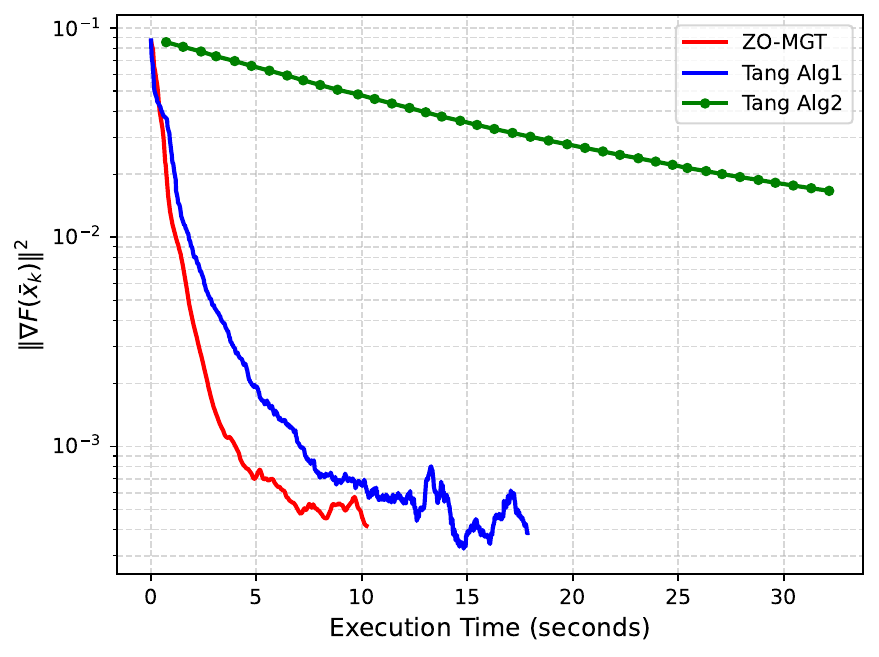}\label{fig:grad_time}}
    \vspace{0.5em} 
    \\
    \subfloat[Consensus Error vs Time]{\includegraphics[width=0.9\linewidth]{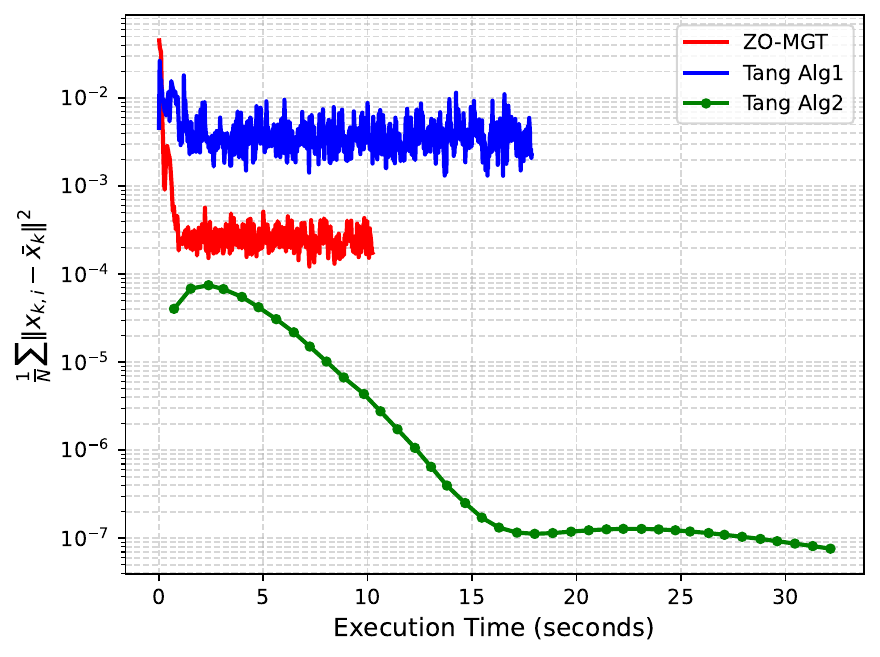}\label{fig:cons_time}}
    \caption{Practical wall-clock time efficiency. }
    \label{fig:exp_time}
\end{figure}

\subsubsection{Computational Efficiency }

Fig. \ref{fig:exp_time} translates iteration complexity into actual wall-clock time, revealing the profound practical superiority of ZO-MGT. We first observe the performance of Tang Alg 2. Although it exhibits rapid per-iteration convergence and tight consensus in terms of iterations, its massive query overhead, requiring a full coordinate-descent estimator with $2d = 248$ evaluations per step, creates a severe temporal bottleneck. Consequently, the $\mathcal{O}(d)$ per-iteration query dependence imposes a substantial computational burden, significantly degrading its convergence efficiency with respect to wall-clock time.

With Tang Alg 2 practically bottlenecked by query costs, the temporal comparison between the two lightweight random-direction methods becomes crucial. Crucially, ZO-MGT demonstrates a distinct efficiency advantage over Tang Alg 1 (Fig. \ref{fig:exp_time}(a)), which is attributed to two primary factors that align with realistic black-box deployment constraints:

First, from the perspective of query efficiency, ZO-MGT utilizes a forward-difference estimator, which inherently evaluates the exact current objective state $f_i(x_{k,i})$ alongside a single perturbed state. This allows the algorithm to seamlessly reuse the current state evaluation for global loss monitoring and convergence tracking without additional cost, maintaining a strict budget of two queries per iteration. In contrast, Tang Alg 1 employs a central-difference estimator that evaluates two perturbed states ($f_i(x \pm \mu u)$). Under a rigorous black-box model where the true objective value is not a byproduct of the gradient estimation, an auxiliary third function evaluation is mandatory at each step solely for state monitoring and performance logging. This 50\% increase in query overhead per iteration significantly inflates the actual execution time of Tang Alg 1.

Second, from a sampling-overhead perspective, the Rademacher perturbations $u \in \{-1, 1\}^d$ driving ZO-MGT are computationally more lightweight to generate than the spherical perturbations in Tang Alg 1. While the latter requires drawing continuous standard Gaussian vectors followed by costly non-linear floating-point operations for $d$-dimensional Euclidean normalization ($u = g / \|g\|$), Rademacher sampling merely relies on generating discrete random signs. When compounded over thousands of iterations and multiple agents, these algorithmic advantages ensure that ZO-MGT achieves superior consensus and stationarity vastly faster in software simulations. Furthermore, the discrete nature of Rademacher perturbations uniquely positions ZO-MGT for extreme acceleration via bitwise operations on specialized hardware, e.g., FPGAs, which remains a highly promising direction for future deployment.

\subsubsection{Verification of the Theoretical Suppression Rate}

To investigate the impact of momentum-based variance reduction and to numerically validate the theoretical bound derived in Theorem \ref{thm:convergence}, we conduct a sensitivity analysis on the momentum factor $\beta$. We evaluate the consensus performance of ZO-MGT across a spectrum of values $\beta \in \{0, 0.5, 0.8, 0.9, 0.98\}$, where $\beta=0$ effectively reduces the algorithm to a vanilla zeroth-order gradient tracking scheme without variance reduction.

\begin{figure}[h]
    \centering
    \subfloat[Consensus Error vs Iterations]{\includegraphics[width=0.9\linewidth]{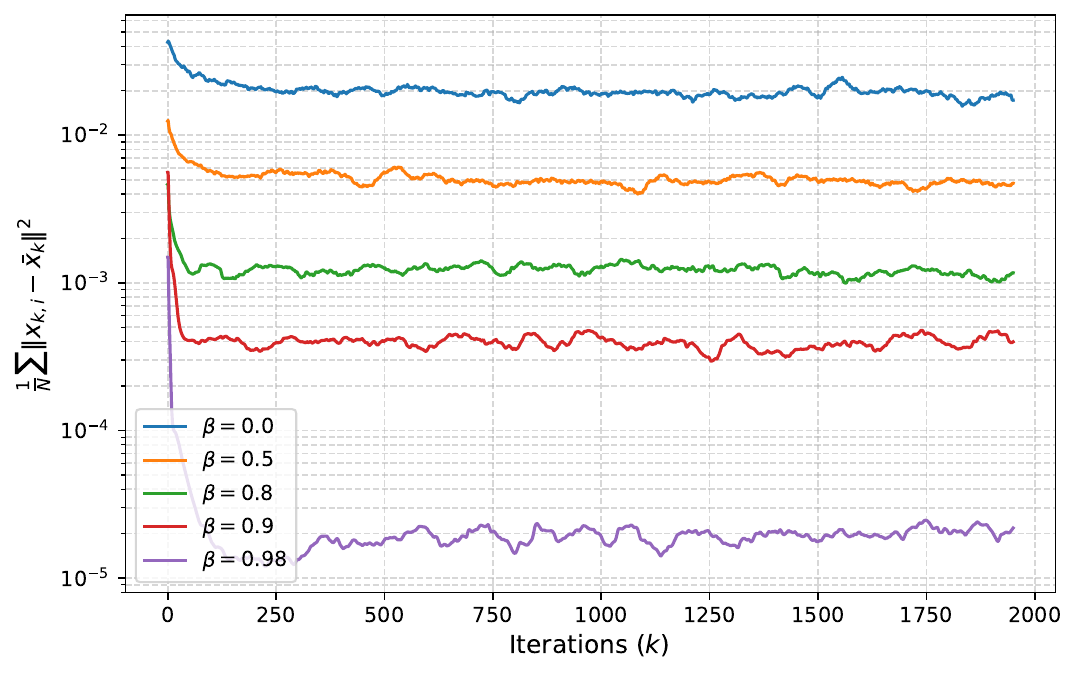}\label{fig:beta_conv}}
    \hfill
    \subfloat[Log-Log Scaling Analysis]{\includegraphics[width=0.9\linewidth]{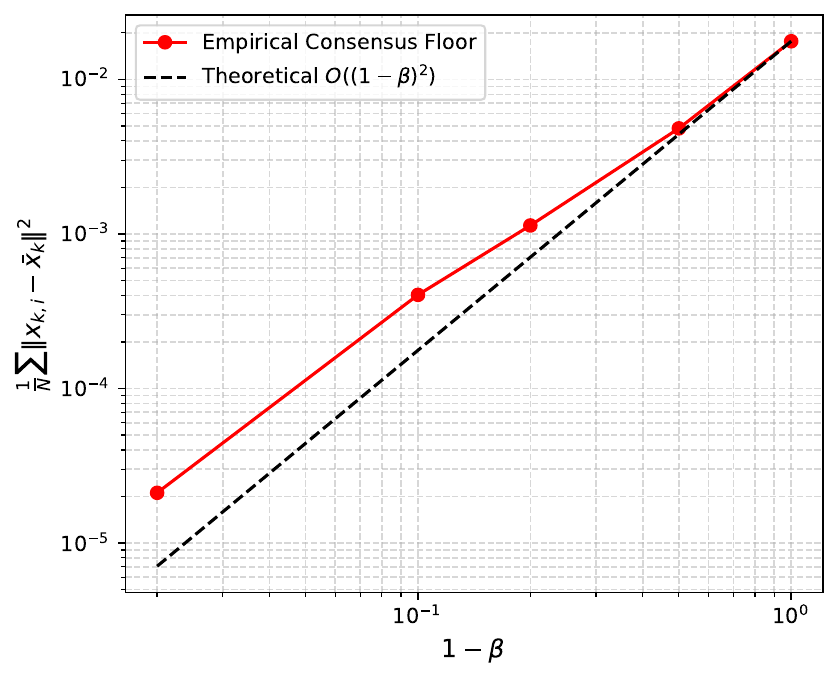}\label{fig:beta_loglog}}
    \caption{Sensitivity analysis of the momentum factor $\beta$. }
    \label{fig:beta_analysis}
\end{figure}

Fig. \ref{fig:beta_analysis}(a) illustrates the evolution of the consensus error floor as a function of $\beta$. 
It is observed that increasing $\beta$ from $0$ to $0.98$ monotonically depresses the steady-state error floor. 
While a larger $\beta$ marginally decelerates the initial transient convergence due to increased tracking lag, it systematically improves the asymptotic precision of the algorithm.
Specifically, the error floor for $\beta=0.98$ is nearly three orders of magnitude lower than that of the vanilla GT ($\beta=0$), highlighting the critical role of momentum as an implicit variance reducer.

Fig. \ref{fig:beta_analysis}(b) presents a log-log scaling analysis between the steady-state consensus error and the term $(1-\beta)$. The empirical data points (red markers) align remarkably well with the theoretical $O((1-\beta)^2)$ reference line (black dashed line). The scaling trend of the empirical results follows the predicted $O((1-\beta)^2)$ scaling, confirming that the error floor decreases quadratically as $\beta$ approaches 1.This result perfectly corroborates our Lyapunov analysis and justifies the integration of momentum-based variance reduction into decentralized zeroth-order frameworks.

\section{Conclusion}
\label{sec:conclusion}

In this paper, we proposed ZO-MGT, a query-efficient distributed zeroth-order algorithm for non-convex multi-agent optimization under severe data heterogeneity. By integrating two-point randomized gradient estimation, momentum variance reduction, and gradient tracking, ZO-MGT effectively mitigates the severe variance amplification typical in derivative-free tracking methods. We theoretically proved an $\mathcal{O}(1/T)$ convergence rate and mathematically demonstrated that a large momentum factor suppresses the heterogeneity-induced error floor at a quadratic rate of $\mathcal{O}((1-\beta)^2)$. ZO-MGT requires exactly two function queries per iteration through Rademacher perturbations. This fixed query budget avoids the $\mathcal{O}(d)$ complexity of existing methods, resulting in faster wall-clock convergence and tighter network consensus. Future work will explore distributed adaptive momentum methods and deploy ZO-MGT on specialized hardware accelerators to actualize the bitwise execution advantages of discrete sampling.

\nocite{*}
\bibliographystyle{IEEEtran}
\bibliography{ZO-MGT}

\begin{IEEEbiography}
	[{\includegraphics[width=1in,height=1.25in,clip,keepaspectratio]{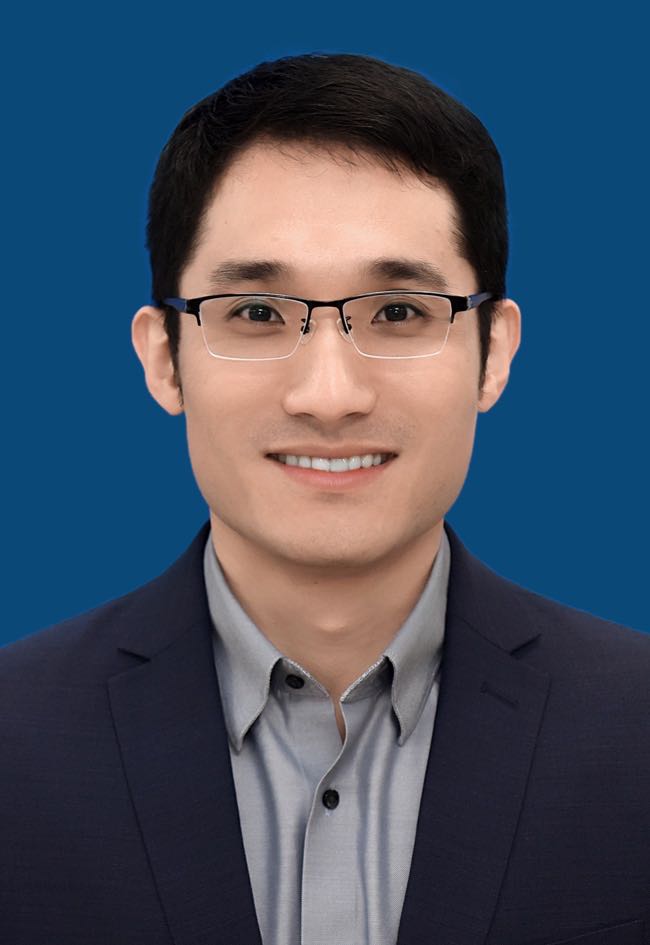}}]{Yanxu Su} (Member, IEEE)
	received his B.E. and M.E. degrees from the College of Automation Engineering, Nanjing University of Aeronautics and Astronautics, Nanjing, China, in Control Engineering in 2012 and 2015, respectively, and the Ph.D. degree from Southeast University in Control Science and Engineering in 2021. He was a visiting Ph.D. student with the Department of Mechanical Engineering, University of Victoria, Victoria, British Columbia, Canada, from 2018 to 2019. He is currently an Associate Professor with the School of Artificial Intelligence, Anhui University, Hefei, China. 
	
	His research interests include distributed optimization, multi-agent systems, model predictive control, etc.
\end{IEEEbiography}
\vspace{-10mm}
\begin{IEEEbiography}[{\includegraphics[width=1in,height=1.25in,clip,keepaspectratio]{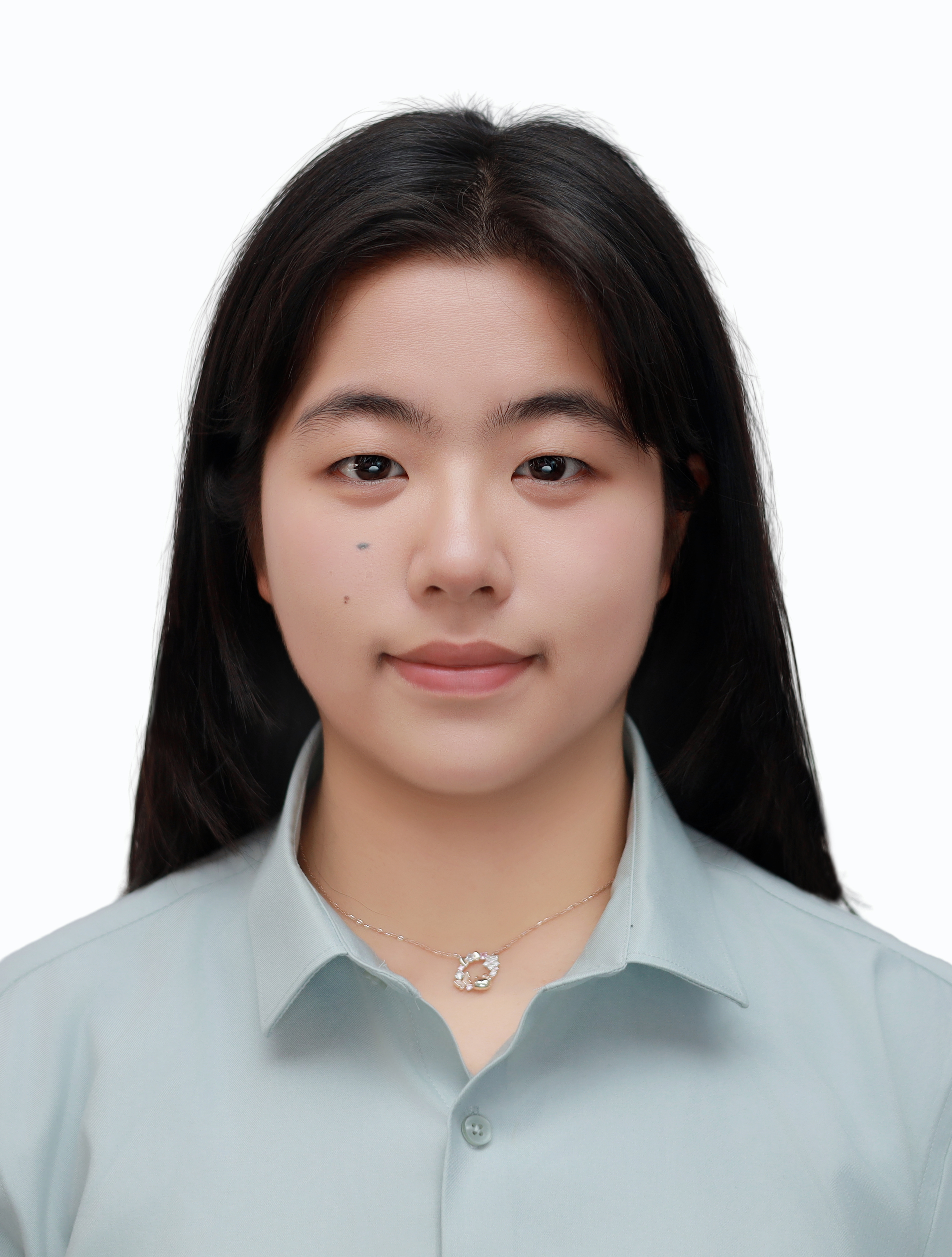}}]{Xiaorui Tong}
	received her B.E. degree from the College of Science, China Three Gorges University, Yichang, China, in Information and Computing Science in 2020 and 2024. She is currently pursuing the M.E. degree in Intelligence Science and Technology with Anhui University, Hefei, China.
	
	Her main research directions are distributed optimization and zeroth-gradient optimization.
\end{IEEEbiography}
\vspace{-10mm}
\begin{IEEEbiography}[{\includegraphics[width=1in,height=1.25in,clip,keepaspectratio]{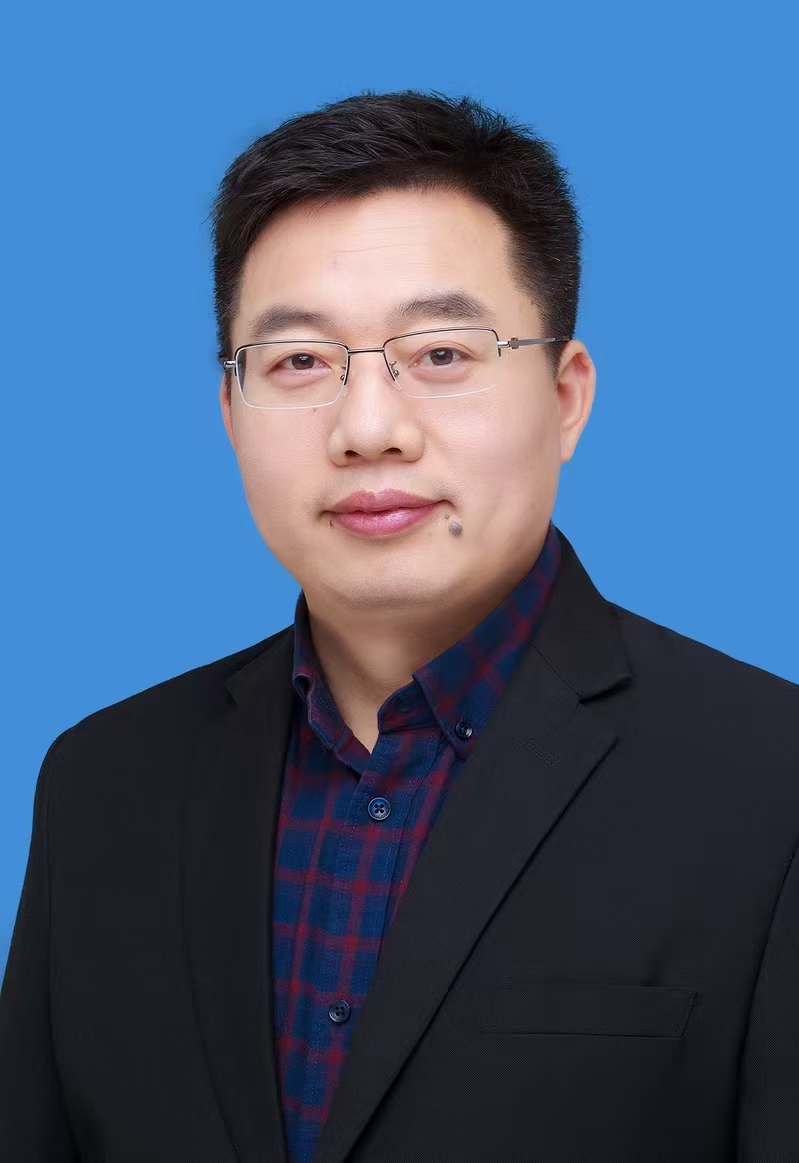}}]{Changyin Sun}(Senior Member, IEEE)
	received the bachelor’s degree from the College of Mathematics, Sichuan University, Chengdu, China, in 1996, and the M.S. and Ph.D. degrees in electrical engineering from Southeast University, Nanjing, China, respectively, in 2001 and 2004. He is currently a Professor with the School of Artificial Intelligence, Anhui University, Hefei, China.
	
	He is the Associate Editor of \textit{IEEE Transactions on Neural Networks and Learning Systems}, \textit{IEEE Neural Processing Letters}, and \textit{IEEE/CAA Journal of Automatica Sinica}. His research interests include intelligent control, flight control, pattern recognition, optimal theory, etc.
\end{IEEEbiography}
\end{document}